\documentclass[reqno]{amsart}
\usepackage[utf8]{inputenc}
\usepackage{fullpage}
\usepackage{amsmath}
\usepackage{amsfonts}
\usepackage{amssymb}
\usepackage{amsthm}
\usepackage{hyperref}
\usepackage{bm}
\usepackage{color}
\usepackage{graphicx}
\usepackage{subcaption}

\newtheorem{theorem}{Theorem}[section]
\newtheorem*{theorem*}{Theorem}
\newtheorem{lemma}[theorem]{Lemma}
\newtheorem{proposition}[theorem]{Proposition}
\newtheorem{corollary}[theorem]{Corollary}
\newtheorem{conjecture}[theorem]{Conjecture}

\newtheorem{example}[theorem]{Example}

\numberwithin{equation}{section}

\theoremstyle{definition}
\newtheorem{definition}[theorem]{Definition}
\newtheorem{remark}[theorem]{Remark}

\theoremstyle{plain}

\newcommand{\R}{\mathbb{R}}

\newcommand{\Z}{\mathbb{Z}}

\DeclareMathOperator{\rev}{rev}
\newcommand{\tcb}[1]{\textcolor{blue}{#1}}

\newcommand{\cF}{\mathcal{F}}

\DeclareMathOperator{\cpc}{Cap}

\DeclareMathOperator{\CT}{CT}
\DeclareMathOperator{\vol}{vol}

\DeclareMathOperator{\relint}{rel\,int}

\DeclareMathOperator{\aff}{aff}
\DeclareMathOperator{\perm}{perm}

\title{Capacity bounds on integral flows \\ and the Kostant partition function}

\author[Leake]{Jonathan Leake}
\address[J. Leake]{Department of Combinatorics and Optimization, University of Waterloo, Canada}
\email{jonathan@jleake.com}
\urladdr{https://www.jleake.com}

\author[Morales]{Alejandro H. Morales}
\address[A. H. Morales]{LACIM, Département de Mathématiques, Universit\'e du Qu\'ebec \`a Montr\'eal, Canada} 
\address{Department of Mathematics and Statistics, UMass Amherst, Amherst, U.S.A.}
\email{morales\_borrero.alejandro@uqam.ca}
\urladdr{https://ahmorales.combinatoria.co}

\date{June 2024}

\begin{document}

\begin{abstract}
    The type $A$ Kostant partition function is an important combinatorial object with various applications: it counts integer flows on the complete directed graph, computes Hilbert series of spaces of diagonal harmonics, and can be used to compute weight and tensor product multiplicities of representations. In this paper we study asymptotics of the Kostant partition function, improving on various previously known lower bounds and settling conjectures of O'Neill and Yip.  Our methods build upon recent results and techniques of Br\"and\'en-Leake-Pak, who used Lorentzian polynomials and Gurvits' capacity method to bound the number of lattice points of transportation and flow polytopes. Finally, we also give new two-sided bounds using the Lidskii formulas from subdivisions of flow polytopes.
\end{abstract}

\maketitle

\section{Introduction}

Integer flows on networks are very important objects in optimization, combinatorics, and representation theory. In the latter context, the number of integer flows on a directed complete graph is also known in Lie theory as the {\em Kostant partition function}.  Many important quantities in representation theory like weight multiplicities (like {\em Kostka numbers}) and tensor product multiplicities (like the {\em Littlewood--Richardson coefficients}) can be expressed in terms of this function \cite{Humphreys}. In this paper, we give new lower bounds on the Kostant vector partition function that improve on previously known bounds.
To do this, we utilize recent lower bounds on the number of {\em contingency tables} given in \cite{BLP}.
Contingency tables are the lattice points of {\em transportation polytopes}, and since flow polytopes can be seen as faces of transportation polytopes, we are able to adapt the results of \cite{BLP} to our context.
The bounds of \cite{BLP} come via lower bounds on the coefficients of certain (denormalized) {\em Lorentzian polynomials} \cite{BH,AGSVI,G09} and their associated generating series, in terms of their {\em capacity} \cite{G06}.
Our main contribution is then new explicit estimates of the capacity of these generating series via an associated {\em flow entropy} quantity, which lead to our new lower bounds.

Let $\bm{N}=(N_0,\ldots,N_{n-1},-\sum_i N_i)$ where $N_i \in \mathbb{Z}$ and let $G$ be a directed acyclic connected graph with vertices $\{0,1,\ldots,n\}$. We denote by $\mathcal{F}_{G}(\bm{N})$ the flow polytope of $G$ with netflow $\bm{N}$.  When $G$ is the complete graph $k_{n+1}$, we denote by the flow polytope by $\mathcal{F}_n(\bm{N}):=\mathcal{F}_{k_{n+1}}(\bm{N})$. We are interested in the number $K_n(\bm{N})$ of lattice points of the flow polytope $\mathcal{F}_n(\bm{N})$, i.e. the number of integer flows of the complete graph $k_{n+1}$ with netflow $\bm{N}$. This is called the {\em Kostant vector partition function} since it has an interpretation in the representation theory of Lie algebras: it is the number of ways of writing $\bm{N}$ as an $\mathbb{N}$-combination of the vectors $\bm{e}_i-\bm{e}_j$ for $1\leq i <j \leq n+1$, the type $A$ positive roots.  

The function $K_n(\bm{N})$ is a piecewise-polynomial function on the parameters $N_i$\footnote{Note that this does not imply polynomial bounds on $K_n(\bm{N})$, since the total degree of the polynomials can depend on the length $n$ of the vector (see Section~\ref{sec: lidskii}).} \cite{SturmVPF} with a complex chamber structure \cite{BMP,BBCV}. Moreover, computing the number of lattice points of $\mathcal{F}_G(\bm{N})$ in general is a
computationally hard
problem \cite{B-SDLV}, and there are special cases that are important and give
surprising answers. We list a few of these, see Section~\ref{sec: previous bounds} for more details.
\begin{itemize}
\item[(i)] $K_{n}(1,0,\ldots,0,-1)=2^{n-1}$, however no general formula is known for $a_n(t):=K_{n}(t,0,\ldots,0,-t)$ which is the Ehrhart polynomial of the flow polytope $\mathcal{F}_{n}(t,0,\ldots,0,-t)$. 
Chan--Robbins--Yuen \cite{CRY} showed that for $t \in \mathbb{N}$, the  sequence $(a_n(t))_{n\geq 0}$  satisfies a linear recurrence of order $p(t)$, the number of integer partitions of $t$. 
Moreover, the following bounds are known for $a_n(t)$ 
\begin{equation}
    \label{eq:bound Ehrhart CRY}
\prod_{1\leq i<j\leq n} \frac{2t+i+j-1}{i+j-1} \geq a_n(t) \geq (t+1)^{n-1}.
\end{equation}
The lower bound follows from elementary methods, and improving this bound with more interesting techniques has proven elusive. The upper bound is more subtle and it follows from containment of  $\mathcal{F}_n(1,0,\ldots,0,-1)$ in another polytope related to \emph{alternating sign matrices} \cite{MMS}.
\item[(ii)] $b_n:=K_{n}(1,2,\ldots,n,-\binom{n+1}{2}) = \prod_{i=1}^n C_i$ where $C_i = \frac{1}{i+1}\binom{2i}{i}$ is the $i$th Catalan number. This
  happens to be the volume of  the $\binom{n}{2}$-dimensional polytope $\mathcal{F}_{n}(1,0,\ldots,0,-1)$ \cite{CRY,BV,Z} and thus leading term in $t$ of $\binom{n}{2}!\cdot a_n(t)$.
\item[(iii)] $c_n:=K_{n}(1,1,\ldots,1,-n)$, this value counts the number of $n\times n$ {\em Tesler
  matrices} \cite{HagTesler,AGHRS} which are of interest in the study of the space of {\em diagonal harmonics} $DH_n$ which has dimension $(n+1)^{n-1}$ (see \cite{Hagbook}), the number of rooted forests on vertices $[n]=\{1,2\ldots,n\}$. Indeed, there is a formula of Haglund  \cite{HagTesler} for the Hilbert series of this space as an alternating sum over the integer flows counted in  $b_n$. No simple formula is known for $b_n$ but from the connection to $DH_n$ and computational evidence \cite[\href{https://oeis.org/A008608}{A008608}]{oeis} it is expected that eventually $c_n> (n+1)^{n-1}$. This is a special case of \cite[Conj. 6.5]{JON}.
 In an effort to show this lower bound, 
  O'Neill \cite{JON} found the following bounds for $c_n$,
 \begin{equation} \label{eq: bounds tesler jason}
2^{\binom{n-2}{2}-1}\cdot 3^n \geq c_n \geq (2n-3)!!.
\end{equation}
Note that $(2n-1)!! \sim \sqrt{2} \cdot  (2/e)^n n^n$. In 2016 Pak (private communication) asked
  whether $c_n$ is $e^{\Theta(n^2)}$ and in 2019 Yip  conjectured (private communication) that for all $n\geq 0$, $c_n$ is at least as big as the number $f_n \sim e^{1/2}\cdot n^{n-2}$ of forests on vertices $[n]$ \cite[\href{http://oeis.org/A001858}{A001858}]{oeis}, which is also the number of lattice points of the permutahedron $\Pi_n$.
  \item[(iv)] $d_n:=K_n(2\rho)$ where $2\rho = (n,n-2,n-4,\ldots,-n+4,-n+2,-n)$ is the sum of all positive roots in type $A_n$ \cite[\href{https://oeis.org/A214808}{A214808}]{oeis}. The quantity $d_n$ gives the dimension of the zero weight space of a certain \emph{Verma module} \cite{Humphreys} and the problem of giving bounds for this quantity was raised in \cite{mathOverflow}. O'Neill obtained in \cite{JON2} the following bound for $K_n(\rho)$ when $n=2k+1$,
  \begin{equation} \label{eq: bound jason rho t=1}
  d_n \geq 3^{k^2-k-1}.
  \end{equation}
  He also gave a bound for $K_n(t\cdot 2\rho)$ (see Proposition~\ref{prop:bounds jason 2rho}).
\end{itemize}

These cases suggest studying the asymptotic behavior of the Kostant partition function. In this paper we obtain the following improvements for the all cases mentioned above\footnote{After the paper was finished, Anne Dranowski kindly informed us that in a subsequent preprint \cite[Thm. 2.5, Cor. 2.7]{BBM} of Ivan Balashov, Constantine Bulavenko, and Yaroslav Molybog, show with different methods that $\log a_n(t) \sim C_1 n\sqrt{t}$ and  $\log c_n \sim C_2 n^{3/2}$  for  universal constants $C_1$ and $C_2$ (June 13, 2024, personal communication).}. 

\begin{theorem} \label{thm:intro-t00-case}
    Fix an integer $t > 0$ and let $\bm{N} = (t,0,0,\ldots,0,-t)$. Then for $a_n(t):=K_n(\bm{N})$ we have 
    \[
       \log_2 a_n(t) \geq \frac{n}{2}  \log_2^2 t - O(n \log_2 t).
    \]
    The big-O notation is with respect to $n$ (with parameter $t$ fixed), and the implied constant is independent of $n$ and $t$.
\end{theorem}

This improves over the bound $\log_2 a_n(t) \geq n \log_2(t+1) - O(\log_2 t)$ from the lower bound in \eqref{eq:bound Ehrhart CRY} by an extra factor of $\frac{1}{2} \log_2 t$. 

For the Tesler case $\bm{N}=(1,1,\ldots,1,-n)$, we have the following lower bound.

\begin{theorem}
    Let $\bm{N}=(1,1,\ldots,1,-n)$. Then for $c_n:=K_n(\bm{N})$ we have
    \[
        \log c_n \geq \frac{n}{4}  \log^2 n - O(n \log n).
    \]
    Furthermore, $c_n \geq (n+1)^{n-1}$ for $n\geq 3000$.
\end{theorem}

This bound beats (asymptotically) all previously known bounds \eqref{eq: bounds tesler jason} and proves Yip's and O'Neill's conjectures mentioned above, for large enough $n$, since $f_n \leq (n+1)^{n-1}$. This bound is the first improvement beyond $(n+1)^{n-1}$ towards answering  Pak's question.

As a comparison, in the case $\bm{N}=(1,2,\ldots,n,-\binom{n+1}{2})$ where $b_n$ has a closed formula and thus  $\log b_n = n^2 \log2 - O(n\log n)$ \cite[Lemma 7.4]{MoralesShi}, our methods give the lower bound $\log b_n \geq \frac12 n^2 - O(n \log n)$. For all polynomial growth cases $N_k\geq a\cdot k^p$ we also obtain a general bound. (We also obtain better bounds for more specific cases; see Section \ref{sec:flow-counting-lower-bounds}.)

\begin{theorem} \label{thm:intro-main-pos-flow-vec}
    Fix $a > 0$ and $p \geq 0$ and suppose $N_k \geq a \cdot k^p$ for all $0 \leq k \leq n-1$. Then
    \[
        \log K_{n}({\bm{N}}) \,\gtrsim\,
        \begin{cases}
            n^2 \log n \cdot \left(\frac{p-1}{2}\right) & p > 1 \\
            n^2 \cdot \left(\frac{1}{2} \log\left(\frac{a}{2}\right) + 2 - 2\log 2\right), & p = 1, a > 2 \\
            n^2 \cdot \left(a - a \log 2)\right), & p = 1, a \leq 2 \\
            n^{p+1} \log^2 n \cdot \left(\frac{a(1-p)^2}{4(p+1)}\right) & p < 1
        \end{cases}.
    \]
    The $\gtrsim$ symbol means the expressions given above essentially give the leading term of the actual lower bounds we obtain. See Theorem \ref{thm:general_positive_lower_bound} for the formal statement of this result.
\end{theorem}

The phase transitions in the above lower bounds are possibly interesting, but we cannot tell whether or not they are artifacts of our proof strategy. See Section \ref{sec:fin-rem-poly-growth} for further discussion.

The final case we discuss is that of $\bm{N} = t \cdot 2\rho$. This case does not fit with the previous cases in the sense that some entries of $\bm{N}$ are negative, but we are still able to apply our methods to achieve the following lower bound.

\begin{theorem}
    Fix an integer $t > 0$ and let $\bm{N} = t \cdot 2\rho(n) = t \cdot (n, n-2, n-4, \ldots, -n+2, -n)$. Then for $d_n(t):=K_n(\bm{N})$ we have
    \[
        \log d_n(t) \geq \frac{n^2}{2} \log \left(\frac{(1+t)^{1+t}}{t^t}\right) - O(n \log(nt)).
    \]
    The implied constant is independent of $t$.
\end{theorem}

This bound improves over the results of O'Neill for all integers $t > 0$ (see above and Proposition \ref{prop:bounds jason 2rho}), including the important case of $t=1$ where we improve the leading-term constant for $\log d_n = \log d_n(1)$ from $\frac{\ln 3}{4}$ to $\ln 2$.

Finally, we prove bounds in various other specific cases using the same methods, and these are collected in Section \ref{sec:flow-counting-lower-bounds}.

\subsection*{Methodology}

To state our main result we need the following notation. Given some flow $(f_{ij})_{0 \leq i < j \leq n} = \bm{f} \in \mathcal{F}_n(\bm{N})$ and letting
\(
    h(t) := (t+1) \log (t+1) - t \log t,
\)
we define the \emph{flow entropy} of $\bm{f}$ via
\begin{equation} \label{eq:flow-entropy}
    \mathcal{H}(\bm{f}) := \sum_{0 \leq i < j \leq n} h(f_{ij}) + \sum_{0 < j < n} h\left(\sum_{0 \leq i < j} (N_i - f_{ij})\right).
\end{equation}
Note that $\mathcal{H}(\bm{f})$ is a concave function on $\mathcal{F}_n(\bm{N})$. With this, we can now state the main technical result we use to prove our bounds.

\begin{theorem}[{flow version of \cite[Lemma 2.2]{Bar10}, \cite[Lemma 5]{Bar12}, \cite[Thm. 3.2]{BLP}}] \label{main-tech-entropy}
    Let $\bm{N} = (N_0,N_1,\ldots,N_n)$ be an integer vector such that $\mathcal{F}_n(\bm{N})$ is non-empty, and let $s_k = \sum_{j=0}^k N_j$ for all $k$. Letting $K_n(\bm{N})$ denote the number of integer points of $\mathcal{F}_n(\bm{N})$, we have
    \[
        \sup_{\bm{f} \in \mathcal{F}_n(\bm{N})} e^{\mathcal{H}(\bm{f})} \,\geq\, K_n(\bm{N}) \,\geq\, \frac{\max_{0 \leq k \leq n-1}\left\{e^{h(s_k)}\right\}}{\left(\prod_{k=0}^{n-1} e^{h(s_k)}\right)^2} \sup_{\bm{f} \in \mathcal{F}_n(\bm{N})} e^{\mathcal{H}(\bm{f})}.
    \]
\end{theorem}

Results similar to Theorem~\ref{main-tech-entropy} have appeared in various contexts before, as suggested by the cited references. That said, we give a new and streamlined proof technique for this fact in Section~\ref{sec:dual-capacity}.

Using Theorem \ref{main-tech-entropy}, we can obtain explicit lower bounds on $K_n(\bm{N})$ for a given flow vector $\bm{N}$ by computing
\begin{equation} \label{eq:flow-entr-expr}
    \frac{\max_{0 \leq k \leq n-1}\left\{e^{h(s_k)}\right\}}{\left(\prod_{k=0}^{n-1} e^{h(s_k)}\right)^2} \cdot e^{\mathcal{H}(\bm{f})}
\end{equation}
for a well-chosen $\bm{f} \in \mathcal{F}_n(\bm{N})$. The $\bm{f}^\star$ which optimizes $\mathcal{H}$ has no closed-form formula in general, and thus some heuristic must be used to choose $\bm{f}$ which yields good bounds. The potential problem with this approach is that asymptotic formulas for $K_n(\bm{N})$ may have phase transitions (see \cite{LP22,DLP20} for examples of this in the context of contingency tables); that is, it is possible that similar values of $\bm{N}$ can lead to different asymptotics. This means that a too-simple heuristic leading to a general formula for a lower bound on $K_n(\bm{N})$ is unlikely to give a high quality bound. 

With this in mind, we devise a heuristic for choosing $\bm{f}$ which is complicated enough to hopefully allow for good bounds, but simple enough to be applicable to a wide range of values of $\bm{N}$. Specifically, we choose $\bm{f}$ to be the average of the vertices of $\mathcal{F}_n(\bm{N})$. On the one hand, counting and computing the average of the vertices of a given flow polytope can be non-trivial in general. On the other, we demonstrate the quality of this choice by proving lower bounds in various cases which are asymptotically better than all previously known bounds.

Finally, once we have our choice of $\bm{f}$, we extract explicit asymptotics from the flow entropy expression (\ref{eq:flow-entr-expr}) evaluated at $\bm{f}$. This last step, while elementary, requires some not-so-trivial analysis of the entropy function via the Euler-Maclaurin formula.

\begin{figure}
\[
\includegraphics{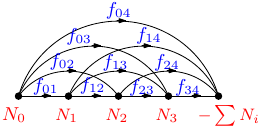} 
\quad 
\raisebox{20pt}{$\to$} 
\quad  
\includegraphics[]{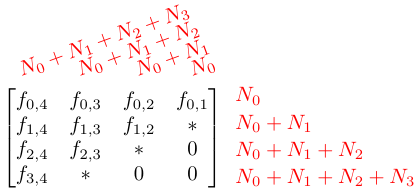}
\]
\caption{Representation of an integer flow of a complete graph and a transportation matrix. The entries marked with $*$ are determined from the rest.}
\end{figure}

\subsection*{A result for intuition}

Theorem \ref{main-tech-entropy} above suggests that a well-chosen $\bm{f} \in \mathcal{F}_n(\bm{N})$ can produce good lower bounds on $K_n(\bm{N})$, and the improved bounds we are able to prove in this paper perhaps demonstrate this for some particular cases. We now state a result which demonstrates this more generally, and offers some more formal evidence why we expect the ideas of this paper to yield good bounds on the number of flows (see Section \ref{sec:asymptotics-max-flow-entropy} for the proof).

\begin{theorem} \label{thm:matrix-choice-asymptotics}
    Let $\bm{N}^* = (N_0,N_1,N_2,\ldots)$ be an infinite sequence of positive integers which has at most polynomial growth, and let $K_n(\bm{N}) = K_n(N_0,\ldots,N_{n-1},-\sum_i N_i)$ and $\mathcal{F}_n(\bm{N}) = \mathcal{F}_n(N_0,\ldots,N_{n-1},-\sum_i N_i)$ for all $n$. Then the maximum flow entropy asymptotically approximates $\log K_n(\bm{N})$. That is, as $n \to \infty$ we have
    \[
        \frac{\log K_n(\bm{N})}{\sup_{\bm{f} \in \mathcal{F}_n(\bm{N})} \mathcal{H}(\bm{f})} \to 1.
    \]
    
\end{theorem}

That is, there is some choice of flows (dependent on $n$) which produces the correct asymptotics in the $\log$. This perhaps makes the problem simpler for the positive polynomial growth case of Theorem \ref{thm:matrix-choice-asymptotics}: instead of counting lattice points of polytopes, we just need to find choices of (not necessarily integer) flows with high entropy. We also remark that $\mathcal{H}(\bm{f})$ in Theorem \ref{thm:matrix-choice-asymptotics} can be replaced by the more standard geometric entropy: $\mathcal{H}_g(\bm{f}) = \sum_{0 \leq i < j \leq n} h(f_{ij})$.

\subsection*{Bounds for other regimes}

Finally, in the results above we mainly consider the individual entries of the flow vector $\bm{N}$ to be constant with respect to $n$. For example, in the case of $a_n(t)$ in Theorem~\ref{thm:intro-t00-case}, we fix $t$ and bound the asymptotics with respect to $n$. However, there are other regimes where bounds are desirable; for example, $t$ may itself be a function of $n$.

To handle cases like this, we use a different technique. Specifically, we use a positive formula for $K_n(\bm{N})$ called the {\em Lidskii formula} \cite{BV,MM18}  coming from the theory of flow polytopes and related to mixed volumes to give bounds for $a_n(t)$ for $t$ much larger than $n$. 

\begin{theorem} \label{thm:bound kpf using Lidskii}
For $t\geq \frac{n^3}{2}$ we have that 
\[
(n-1)!\cdot \binom{t+n-1}{\binom{n}{2}} \prod_{i=0}^{n-2} C_i \,\geq\,   a_n(t)  \,\geq\, \binom{t+n-1}{\binom{n}{2}} \prod_{i=0}^{n-2} C_i. 
\]
\end{theorem}

These concrete bounds are reasonable compared to the leading coefficient $\prod_{i=1}^{n-2} C_i/\binom{n}{2}!$ in $t$ of $a_n(t)$, the volume of the polytope $\mathcal{F}_n(1,0,\ldots,0,-1)$ \cite{Z}.

The techniques used for these bounds do not fit directly into the overarching methodology discussed above. That said, we include them anyway for completeness, and due to the connection between the Lidskii formula and mixed volumes. Mixed volumes are the coefficients of volume polynomials, which are Lorentzian (see, e.g., \cite{BH}), and thus there may be some further connection between this formula and the entropy-based methodology discussed above. We leave this to future work.

\subsection*{Structure of the paper}

The paper is organized as follows. Section~\ref{sec:background} has background on flow polytopes,  bounds, capacity method on contingency tables, and asymptotics of entropy related functions. 
Section~\ref{sec:dual-capacity} gives our proof of Theorem~\ref{main-tech-entropy}.
Section~\ref{sec:avg-of-vertices} has details on the averages of vertices of flow polytopes.
Section~\ref{sec:flow-counting-lower-bounds} computes the concrete asymptotic lower bounds for all the cases we consider.
Section~\ref{sec: lidskii} gives bounds on the Kostant partition function using the Lidskii formula from the theory of flow polytopes.
Section~\ref{sec:final remarks} has final remarks.
Details of the asymptotic analysis are in the Appendix~\ref{sec:bounds_asymptotics}.

\section{Background and combinatorial/geometric bounds} \label{sec:background}

\subsection{Transportation and flow polytopes} \label{sec:flow poly}

A {\em polytope} $P\subset \mathbb{R}^m$ is a convex hull of finitely many points or alternatively a bounded intersection of finitely many half spaces. The polytopes we consider are {\em integral}, i.e. its vertices have integer coordinates. Two polytopes $P\subset \mathbb{R}^m$ and $Q\subseteq \mathbb{R}^k$ are {\em integrally equivalent} if there is an affine transformation $\varphi:\mathbb{R}^m\to \mathbb{R}^k$ such that restricted to $P$ and $\mathbb{Z}^m \cap \aff(P)$ gives a bijection $Q$ and to $\mathbb{Z}^k \cap \aff(Q)$. Next, we define flow polytopes and transportation polytopes.

\begin{definition}[Transportation polytopes]
Let $\bm{\alpha}=(\alpha_0,\ldots,\alpha_{m-1})$ and $\bm{\beta}=(\beta_0,\ldots,\beta_{n-1})$ be vectors in $\mathbb{Z}_{\geq 0}^n$, the {\em transportation polytope} $\mathcal{T}(\bm{\alpha},\bm{\beta})$ is the set of all $m\times n$ matrices $M=(m_{i,j})$ with nonnegative real entries with row sums and column sums $\bm{\alpha}$ and $\bm{\beta}$. The lattice points of $\mathcal{T}(\bm{\alpha},\bm{\beta})$ are called {\em contingency tables} and we denote the number of such tables by $\CT(\bm{\alpha},\bm{\beta})$.
\end{definition}

The generating function of $\CT(\bm{\alpha},\bm{\beta})$ has the following closed form.
\begin{equation}
\Phi(\bm{x},\bm{y}) :=\sum_{\bm{\alpha},\bm{\beta}} \CT(\bm{\alpha},\bm{\beta})\, \bm{x}^{\bm{\alpha}} \bm{y}^{\bm{\beta}} \,=\, \prod_{i=0}^{m-1} \prod_{j=0}^{n-1} \frac{1}{1-x_iy_j},   
\end{equation}

\begin{definition}[flows and flow polytope]
Given a  a directed acyclic graph $G$ with vertices $\{0,1,\ldots,n\}$ and $m$ edges, and a vector $\bm{N}=(N_0,\ldots,N_{n-1},-\sum_i N_i) \in \mathbb{Z}^{n+1}$, an {\em $\bm{N}$-flow} on $G$ is a tuple $(f_{e})_{e\in E(G)}$ in $\mathbb{R}_{\geq 0}^m$ of values assigned to each edge such that the {\em netflow} on vertex $i$ is $N_i$:
\[
\sum_{e=(i,j) \in E(G)} f_e \quad -\quad  \sum_{e=(k,i) \in E(G)} f_e = N_i.
\]
The {\em flow polytope} $\mathcal{F}_G(\bm{N})$ is the set of $\bm{N}$-flows on $G$. 
\end{definition}

When $G$ is the complete graph $k_{n+1}$, for brevity we denote by $\mathcal{F}_n(\bm{N})$ the flow polytope $\mathcal{F}_G(\bm{N})$. We denote by $K_n(\bm{N}):=\#(\mathcal{F}_n(\bm{N}) \cap \mathbb{Z}^{\binom{n+1}{2}})$ the number of lattice points of $\mathcal{F}(\bm{N})$, which counts the number of integer flows on $k_{n+1}$ with netflow $\bm{N}$. This is called {\em Kostant's vector partition function} since it also counts the number of ways   of writing $\bm{N}$ as a combination of the positive type-$A$ roots $\bm{e}_i-\bm{e}_j$ corresponding to each edge $(i,j)$ where $\bm{e}_i$ is a standard basis vector.

The generating function of $K_n(\bm{N})$ for $\bm{N}\in \mathbb{Z}^{n+1}$ has the following closed form.
\begin{equation} \label{eq:gen series kpf}
\Psi(\bm{z}) :=\sum_{\bm{N}} K_n(\bm{N}) \, \bm{z}^{\bm{N}} \,=\, \prod_{0\leq i<j \leq n} \frac{1}{1-z_iz_j^{-1}}.    
\end{equation}

Flow polytopes can be viewed as faces of transportation polytopes as follows (see \cite[\S 1.3]{MMR}). Given a flow vector $\bm{N}$, define $\bm\alpha = (s_0,s_1,\ldots,s_{n-1})$ and $\bm\beta = \rev(\bm\alpha)=(s_{n-1},\ldots,s_1,s_0)$ where $s_k = \sum_{j=0}^k N_j$ as above. Define a linear injection $\phi: \mathcal{F}_n(\bm{N}) \to \mathcal{T}(\bm\alpha,\bm\beta)$ via
    \begin{equation}\label{eq: def injection flows to transportation matrix}
        \phi: (f_{ij})_{0 \leq i < j \leq n} \mapsto
        \begin{bmatrix}
            f_{0,n} & f_{0,n-1} & f_{0,n-2} & f_{0,n-3} & \cdots & f_{0,3} & f_{0,2} & f_{0,1} \\
            f_{1,n} & f_{1,n-1} & f_{1,n-2} & f_{1,n-3} & \cdots & f_{1,3} & f_{1,2} & \tcb{g_{n-1}} \\
            f_{2,n} & f_{2,n-1} & f_{2,n-2} & f_{2,n-3} & \cdots & f_{2,3} & \tcb{g_{n-2}} & 0 \\
            f_{3,n} & f_{3,n-1} & f_{3,n-2} & f_{3,n-3} & \cdots & \tcb{g_{n-3}} & 0 & 0 \\
            \vdots & \vdots & \vdots & \vdots & \ddots & \vdots & \vdots & \vdots \\
            f_{n-3,n} & f_{n-3,n-1} & f_{n-3,n-2} & \tcb{g_3} & \cdots & 0 & 0 & 0 \\
            f_{n-2,n} & f_{n-2,n-1} & \tcb{g_2} & 0 & \cdots & 0 & 0 & 0 \\
            f_{n-1,n} & \tcb{g_1} & 0 & 0 & \cdots & 0 & 0 & 0 \\
        \end{bmatrix},
    \end{equation}
    where the {\em subdiagonal entries} $\tcb{g_j}$ are chosen so that the row sums and column sums are equal to the entries of $\bm\alpha$ and $\bm\beta$. (Note that these entries are given precisely by $g_j:=\sum_{0 \leq i < j} (N_i - f_{ij})$ for $0 < j < n$). The image of $\phi$ is the set matrices in $\mathcal{T}(\bm\alpha,\bm\beta)$ which have $0$ in all entries of the bottom-right corner of the matrix as specified in the definition of $\phi$. The set $\phi (\mathcal{F}_n(\bm{N}))$ is a face of $\mathcal{T}(\bm{\alpha},\bm{\beta})$ (see \cite[Prop. 1.5]{MMR}). 
    
    By abuse of notation, we will refer to a flow $(f_{ij})$ in $\mathcal{F}_n(\bm{N})$ and its image $\phi(f_{ij})$ in $\mathcal{T}(\bm{\alpha},\bm{\beta})$ interchangeably. 

\begin{proposition}[{\cite[Prop. 1.5]{MMR}}] \label{prop: flow is face of trans poly}
Given $\bm{N}=(N_0,\ldots,N_{n-1},-\sum_i N_i) \in \mathbb{Z}^{n+1}$, and $\bm{\alpha}$, $\bm{\beta}$, and $\phi$ be defined as above, then  $\phi$  is an integral equivalence between $\mathcal{F}_n(\bm{N})$ and a face of $\mathcal{T}(\bm{\alpha},\bm{\beta})$.
\end{proposition}

\subsection{Special examples of flow polytopes}

The following are two examples of flow polytopes that will be of interest and we give their vertex description.

\subsubsection{CRY polytope}

For $\bm{N}=(1,0,\ldots,0,-1)$, the polytope $\mathcal{F}_n(\bm{N})$ is called the {\em {\em Chan-Robbins-Yuen} polytope} \cite{CRY}. This polytope has dimension $\binom{n}{2}$ and $2^{n-1}$ vertices \cite{CRY}. The vertices can be described as follows: they correspond to unit flows along paths on $k_{n+1}$ from the source $0$ to the sink $n+1$. These paths are completely determined by their support on internal vertices in $\{1,\ldots,n-1\}$.   We translate the description of these vertices in the transportation polytope.

\begin{proposition}[] \label{lemma:char vertices CRY}
    For $\bm{N}=(1,0,\ldots,0,-1)$, the vertices of $\mathcal{F}_n(\bm{N})$ are determined by binary strings in the sub-diagonal: $(g_1,\ldots,g_{n-1})\in \{0,1\}^{n-1}$. In particular there are $2^{n-1}$ vertices.

Given a binary string $\bm{g}=(g_1,\ldots,g_{n-1})$, the corresponding  vertex $\bm{g} \mapsto X=(x_{ij})$  given by 
\begin{equation} \label{eq:vertices cry case}
x_{i,j} = \begin{cases}
    (1-g_i)(1-g_{n-j})\prod_{k=i+1}^{n-j-1} g_k & \text{if } i<n-j \\
    g_i & \text{if } i=n-j \\
    0 &\text{otherwise}, 
    \end{cases}
\end{equation}
where $g_0=g_n=0$.
\end{proposition}

Since $\mathcal{F}_n(1,0,\ldots,-1)$ is a $0/1$ polytope, its lattice points are its vertices and so $K_n(1,0,\ldots,0,-1)=2^{n-1}$.

\subsubsection{Generalized Tesler polytope} For ($\bm{N}=(N_0,\ldots,N_{n-1},-\sum_{i=0}^{n-1} N_i)$ where $N_i >0$) $\bm{N}=(1,\ldots,1,-n)$, the polytope $\mathcal{F}_n(\bm{N})$ is called the {\em (generalized) Tesler polytope} \cite{MMR}. This polytope has dimension $\binom{n}{2}$, is simple, and has $n!$ vertices. The vertices can be characterized as follows.

\begin{theorem}[{\cite[Thm. 2.5 \& Cor. 2.6]{MMR}}] \label{lemma:char vertices all ones}
    For $\bm{N}=(N_0,\ldots,N_{n-1},-\sum_{i=0}^{n-1} N_i)$ where $N_i>0$, the vertices of $\mathcal{F}(\bm{N})$ are characterized by flows whose associated matrix has exactly one nonzero upper triangular entry in each row. In particular there are $n!$ vertices.

Given a choice $p$ of an upper triangular entry in each row, the corresponding vertex $p\mapsto X=(x_{ij})$ is defined by 
\begin{equation} \label{eq: vert}
x_{i,j} = \begin{cases}
    N_{i} + \sum_{r=0}^{i-1} x_{r,n-i}  &\text{ if } p_{i,j}=1\\
    0 &\text{ otherwise}.
\end{cases}
\end{equation}
and the sub-diagonal term is $x_{i,n-i}=\sum_{r=0}^{i-1} (N_r-x_{r,n-i})$ for $i=1,\ldots,n-1$.
\end{theorem}

\begin{example}
The CRY polytope $\mathcal{F}_3(1,0,0,-1)$ is $3$-dimensional with vertices/lattice points represented by the following matrices:
\[
\begin{bmatrix}
0 & 0 & 1 \\
0 & 1 & \tcb{0}\\
1 & \tcb{0} & 
\end{bmatrix}, \quad \begin{bmatrix}
1 & 0 & 0 \\
0 & 0 &\tcb{1}\\
0 & \tcb{1} & 
\end{bmatrix}, \quad \begin{bmatrix}
0 & 1 & 0 \\
0 & 0 &\tcb{1}\\
1 & \tcb{0} & 
\end{bmatrix}, \quad \begin{bmatrix}
0 & 0 & 1 \\
1 & 0 &\tcb{0}\\
0 & \tcb{1} & 
\end{bmatrix}.
\]
The Tesler polytope $\mathcal{F}_3(1,1,1,-3)$ is $3$-dimensional with the following seven lattice points of which the first six are its vertices:
\[
\begin{bmatrix}
0 & 0 & 1 \\
0 & 2 & \tcb{0}\\
3 & \tcb{0} & 
\end{bmatrix},  \begin{bmatrix}
0 & 0 & 1 \\
2 & 0 & \tcb{0}\\
1 & \tcb{2} & 
\end{bmatrix},\begin{bmatrix}
0 & 1 & 0 \\
0 & 1 & \tcb{1}\\
3 & \tcb{0} & 
\end{bmatrix},\begin{bmatrix}
0 & 1 & 0 \\
1 & 0 & \tcb{1}\\
2 & \tcb{1} & 
\end{bmatrix},\begin{bmatrix}
1 & 0 & 0 \\
0 & 1 & \tcb{1}\\
2 & \tcb{1} & 
\end{bmatrix},\begin{bmatrix}
1 & 0 & 0 \\
1 & 0 & \tcb{1}\\
1 & \tcb{2} & 
\end{bmatrix}, \begin{bmatrix}
0 & 0 & 1 \\
1 & 1 & \tcb{0}\\
2 & \tcb{1} & 
\end{bmatrix}.
\]
See Figure~\ref{fig:cryandtesler}.
\end{example}

\begin{figure}
\begin{subfigure}[b]{0.45\textwidth}
    \centering
   \raisebox{25pt}{ \includegraphics[scale=0.75]{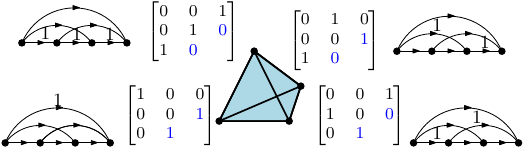}}
    \caption{}
    \label{fig:cry}
\end{subfigure}
\begin{subfigure}[b]{0.5\textwidth}
    \includegraphics[scale=0.7]{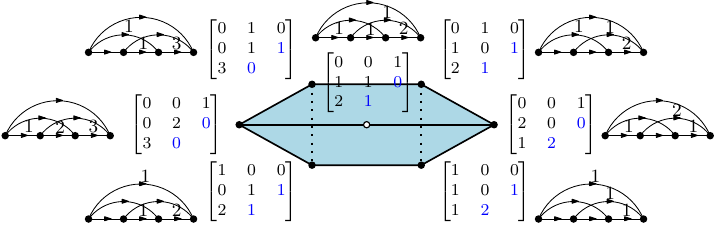}
    \caption{}
    \label{fig:tesler}
\end{subfigure}
\caption{The $3$-dimensional (a) CRY and (b) Tesler polytopes.}
    \label{fig:cryandtesler}
\end{figure}

\subsection{Previous bounds on lattice points of flow polytopes} \label{sec: previous bounds}

In this section we collect previous results and questions about bounds on $K_n(\bm{N})$. See Table~\ref{table: values kpf}.

\begin{proposition}[{e.g. \cite[\S 2.1]{JON}}] \label{prop:generating algorithm}
Let $\bm{N}=(N_0,\ldots,N_{n-1},-\sum_{i=0}^{n-1} N_i)$, then
\begin{equation} \label{eq:gen algorithm}
K_n(N_0,\ldots,N_{n-1},-\sum_i N_i) = \sum_{f} (f_{0,n-1}+1)(f_{1,n-1}+1)\cdots (f_{n-2,n-1}+1),
\end{equation}
where the sum is over integer flows in $\mathcal{F}_{n-1}(N_0,\ldots,N_{n-2},-\sum_i N_i)$. 
\end{proposition}

Given weak compositions $\bm{N}'=(N_0,\ldots,N_{n-1})$ and $\bm{M}'=(M_0,\ldots,M_{n-1})$, we say that $\bm{N}'$ {\em dominates} $\bm{M}'$ if $\sum_{i=0}^k N_i \geq \sum_{i=0}^k M_i$ for every $k=0,\ldots,n-1$ and denote it by $\bm{N}'\unrhd \bm{M}'$.

\begin{proposition} \label{prop:kpf inequality dominance}
Let $\bm{N}=(N_0,\ldots,N_{n-1},-\sum_i N_i)$ and $\bm{M}=(M_0,\ldots,M_{n-1},-\sum_i M_i)$ where $N_i$ and $M_i$ are in $\mathbb{Z}_{\geq 0}$ such that $(N_0,\ldots,N_{n-1}) \unrhd (M_0,\ldots,M_{n-1})$, then $K_n(\bm{N}) \geq K_n(\bm{M})$.    
\end{proposition}

\begin{proof}
Let $\epsilon_i = (\sum_{j=0}^i N_j)-(\sum_{j=0}^i M_j) \geq 0$ for $i=0,\ldots,n-1$, $\bm\alpha=(s_0,s_1,\ldots,s_{n-1})$  and $\bm\beta=(s_{n-1},\ldots,s_1,s_0)$ ($\bm\alpha'=(s'_0,s'_1,\ldots,s'_{n-1})$ and $\bm\beta'=(s'_{n-1},\ldots,s'_1,s'_0)$, respectively) where $s_k=\sum_{j=0}^k N_j$ (for $s'_k=\sum_{j=0}^k M_j$).
Given an integer flow $(f_{ij})$ of $\mathcal{F}_n(\bm{M})$ viewed as a lattice point $\phi(f_{ij})$ in $\mathcal{T}(\bm\alpha,\bm\beta)$, let $(f'_{ij})$ be defined as 
\[
f'_{ij} = \begin{cases}
f_{ij}+\epsilon_i &\text{ if } j=i+1,\\
f_{ij} &\text{ otherwise}.
\end{cases}
\]
Then $(f'_{ij})$ is an integer flow in $\mathcal{F}_n(\bm{N})$, i.e. $\phi(f'_{ij})$ is a lattice point in $\mathcal{T}(\bm\alpha',\bm\beta')$. This map is injective and therefore $K_n(\bm{M}) \leq K_n(\bm{N})$, as desired.
\end{proof}

In particular, the previous result implies a similar inequality when $\bm{N}
'$ is term-wise larger than $\bm{M}'$.

\begin{corollary} \label{cor:kpf inequality term-wise dominance}
Let $\bm{N}=(N_0,\ldots,N_{n-1},-\sum_i N_i)$ and $\bm{M}=(M_0,\ldots,M_{n-1},-\sum_i M_i)$ where $N_i$ and $M_i$ are in $\mathbb{Z}_{\geq 0}$ such that $N_i\geq M_i$ then $K_n(\bm{N}) \geq K_n(\bm{M})$. 
\end{corollary}

\subsubsection{The cases with  closed formulas}

The following cases of $K_n(\bm{N})$ have closed formulas coming from a certain constant term identity due to Zeilberger \cite{Z}, that is a variation of the {\em Morris constant term identity} related to the Selberg integral. For self contained proofs of these product formulas see \cite{BVMorris,MoralesShi}. Let $C_n:=\frac{1}{n+1}\binom{2n}{n}$ denote the $n$th Catalan number and let 
\[
F(t,n):=\prod_{1\leq i<j\leq n} \frac{2t+i+j-1}{i+j-1},
\]
which counts the number of {\em plane partitions} of shape $\delta_n=(n-1,n-2,\ldots,1)$ with entries at most $t$ in $\mathbb{Z}_{\geq 0}$ \cite{proctor}.

\begin{theorem}[{\cite[Prop. 26]{BV} \cite[Eq. (8)]{Mproduct}}] \label{thm:case M_n(t,0,0)}
For a nonnegative integer $t$ we have that 
\begin{equation} \label{eq:case M_n(t,0,0)}
    K_n{(t,t+1,t+2,\ldots,t+n-1,-nt-{\textstyle \binom{n}{2}})} = C_1C_2\cdots C_{n-1} \cdot F(t,n).
\end{equation}
\end{theorem}

\begin{remark}
By a result of Postnikov--Stanley (unpublished) and Baldoni--Vergne \cite{BV}, we have the following relation between the normalized volume of $\mathcal{F}_n(1,0,\ldots,0,-1)$ equals a value of $K_n(\cdot)$,
\begin{equation} \label{eq:prod cat case}
\vol \mathcal{F}_{n+1}(1,0,\ldots,0,-1) = K_n(0,1,2,\ldots,n-2,-{\textstyle \binom{n-1}{2}})=C_1C_2\cdots C_{n-1},
\end{equation}
where the second equality follows by setting $t=1$ in \eqref{eq:case M_n(t,0,0)}. Moreover, the leading term in $t$ of \eqref{eq:case M_n(t,0,0)} gives the normalized volume of  $\mathcal{F}_n(1,1,\ldots,1,-n)$ \cite[Theorem 1.8]{MMR}, \cite[Lemma 2.1]{ZhouLuFu}.
\[
\vol \mathcal{F}_{n+1}(1,1,\ldots,1,-n) = C_1C_2\cdots C_{n-1} \cdot f^{\delta_n},
\]
where $f^{\delta_n}=\binom{n}{2}! 2^{\binom{n}{2}}/\prod_{i=1}^n i!$ is the number of Standard Young tableaux of shape $\delta_n=(n-1,n-2,\ldots,1)$.
\end{remark}

The next two results collect asymptotics and bounds for the special cases and the pieces of the product on the RHS of \eqref{eq:case M_n(t,0,0)}

\begin{proposition}[{\cite[Lemma 7.4]{MoralesShi}}] \label{prop:asympt delta}
    \[
\log K_n{(0,1,2,\ldots,n-1,-{\textstyle \binom{n}{2}})}  = \log C_1C_2\cdots C_{n-1} = n^2 \log 2 - \frac32 n \log n + O(n).
\]
\end{proposition}

\begin{proposition}[{\cite[Proposition 7.5]{MoralesShi}}] \label{prop:asympt n+delta}
\[
\begin{split}
    \log K_n{(n,n+1,n+2,\ldots,2n-2,-n^2-{\textstyle \binom{n-1}{2}})} &= \log(C_1C_2\cdots C_{n-1} F(n,n)) \\
        &= \left(9\log2 - \frac92 \log 3\right)n^2  + O(n\log n).
\end{split}
\]     
\end{proposition}

\begin{proposition}[{\cite[Proposition 3.1]{MPPSchub}}] \label{prop:F-f-bound}
For all integers $n$ and $t$ we have that
\[
0\,\geq\, \log F(t,n)  - (n+t)^2f(t/(n+t))\,\geq\, - 2(t+n),
\]
where $f(x)=x^2\log x - \frac12 (1-x)^2 \log(1-x) - \frac12 (1+x)^2\log(1+x) + 2x\log 2$.
\end{proposition}

\subsubsection{The CRY case  $\bm{N}=(t,0,\ldots,0,-t)$}

\begin{proposition}[{upper bound from \cite{MMS}}] \label{prop: previous bounds CRY case}
For $\bm{N}=(t,0,\ldots,0,-t)$ we have that 
\[
F(t,n) \geq K_n(\bm{N}) \geq (t+1)^{n-1}.
\]
\end{proposition}

\begin{proof}
To show the lower bound we use Proposition~\ref{prop:generating algorithm} and the fact that the smallest product in the sum on the RHS of \eqref{eq:gen algorithm} occurs when the first column is $(f_{0,n-1},\ldots,f_{n-2,n-1})$ is $(0,\ldots,0,t)$ to conclude that $K_n(t,0,\ldots,0,-t)\geq (t+1) \cdot K_{n-1}(t,0,\ldots,0,-t)$. The upper bound comes from \cite[Cor. 5.8]{MMS}.  
\end{proof}

\begin{theorem}[{Chan--Robbins--Yuen \cite[Thm. 1]{CRY}}]
Fix an integer $t>0$ and let $\bm{N}=(t,0,\ldots,0,-t)$,  then $a_n(t)=K_n(t,0^{n-1},-t)$ satisfies a linear recurrence of order $p(t)$, the number of integer partitions of $t$. In particular $a_n(t) \sim c(t)^n$ for some constant $c(t)$.
\end{theorem}

\subsubsection{The Tesler case $\bm{N}=(t,t,\ldots,t,-nt)$}

\begin{proposition}
For $\bm{N}=(t,t,\ldots,t,-nt)$ we have that 
\[
(t+1)^{\binom{n}{2}}\,\geq\, K_n(\bm{N}) \,\geq\, \prod_{i=1}^{n-1} (it+1).
\]
\end{proposition}

\begin{proof}
To show the lower bound we use Proposition~\ref{prop:generating algorithm} and the fact that the smallest and biggest product in the sum on the RHS of \eqref{eq:gen algorithm} occurs when the first column is $(f_{0,n-1},\ldots,f_{n-2,n-1})$ is $(0,\ldots,0,(n-1)t)$ and $(t,\ldots,t)$, respectively. This implies that
\[
t^{n-1}  K_{n-1}(t,t,\ldots,t,-(n-1)t) \,\geq\, K_n(t,t,\ldots,t,-nt) \,\geq\, ((n-1)t+1) \cdot K_{n-1}(t,t,\ldots,t,-(n-1)t).
\]
\end{proof}

A more careful analysis of Proposition~\ref{prop:generating algorithm} was performed by O'Neill to improve the bounds for the case of $\bm{N}=(1,1,\ldots,1,-n)$.

\begin{proposition}[{O'Neill \cite{JON}}]
 \[
2^{\binom{n-2}{2}-1}\cdot 3^n\,\geq\, K_n{(1,1,\ldots,1,-n)} \,\geq\, (2n-3)!! .
\]   
\end{proposition}

Let $\Pi_n=\textup{conv}(w (0,1,\ldots,n-1) \mid w \in \mathfrak{S}_n\}$ be the {\em classical permutahedron}. In 2019, Yip (private communication) asked whether the Tesler polytope $\mathcal{F}_n(\bm{1})$ projects to the permutahedron and conjectured the following weaker statement. Recall that $f_n \sim e^{1/2}\cdot n^{n-2}$ is the number of forests  with vertices $[n]$ \cite[{\href{http://oeis.org/A001858}{A001858}}]{oeis}. This is also the number of lattice points of  $\Pi_n$.

\begin{conjecture}[Yip] \label{thm:Yip's conjecture}
For $\bm{N}=(1,1,\ldots,1,-n)$, we have that $K_n(\bm{N}) \geq f_n$.
\end{conjecture}

Next, we give some other bounds for  the number of lattice points of dilations of the Tesler polytope.

\begin{proposition} \label{prop:bound t to t+i}
For $t\geq 0$ then 
\begin{equation} \label{eq: simple geom bound}
C_1C_2\cdots C_{n-1} \cdot F(t,n) \geq K_n{(t,t,\ldots,t,-nt)} .
\end{equation} 
\end{proposition}

\begin{proof}
By Corollary~\ref{cor:kpf inequality term-wise dominance} we have that 
\[
 K_n{(t,t+1,t+2,\ldots,t+n-1,-nt-{\textstyle \binom{n}{2}})}\geq K_n{(t,t,\ldots,t,-nt)}.
\]
The upper bound follows by the product formula in \eqref{eq:case M_n(t,0,0)} for the  LHS above. 

\end{proof}

\begin{proposition}
\label{prop:bound t to t-n+i}
For $t\geq n-1$ we have that 
\[
    K_n{(t,t,\ldots,t,-nt)} \geq  C_1C_2\cdots C_{n-1}\cdot  F(t-n+1,n). 
    \]
\end{proposition}

\begin{proof}
When $t\geq n-1$, by Corollary~\ref{cor:kpf inequality term-wise dominance} we have that 
\[
K_n{(t,t,\ldots,t,-nt)} \geq K_n{(t-n+1,t-n+2,\ldots,t-1,t,-nt+{\textstyle \binom{n}{2}})}.
\]
The lower bound follows by the product formula in \eqref{eq:case M_n(t,0,0)} for the  RHS above evaluating $t=t-n+1$. 
\end{proof}

\begin{corollary}
\[
n^2(9\log2 - \frac92 \log 3)  + O(n\log n)\geq \log K_n{(n,n,\ldots,n,-n^2)} \geq n^2\log 2  - \frac32 n \log n + O(n).
\]
\end{corollary}

\begin{proof}
By Propositions~\ref{prop:bound t to t+i},~\ref{prop:bound t to t-n+i} for $t=n-1$ we obtain that 
\[
C_1C_2\cdots C_{n-1}\cdot F(n,n)\geq K_n(n,n,\ldots,n,-n^2) \geq C_1C_2\cdots C_{n-1}.
\]
Taking the log and using the bounds in Propositions~\ref{prop:asympt delta}, \ref{prop:asympt n+delta} give the result.
\end{proof}

\begin{table}
\begin{center}
\begin{tabular}{ll} \hline
$\bm{N}$ & $K_n(\bm{N})$ \\ \hline 
$(t,0,\cdots,0,-t)$ & no closed formula \\
$(1,1,\cdots,1,-n)$  & no closed formula \\
$(0,1,2,\cdots,(n-1),-\binom{n}{2})$ & $C_1\cdots,C_{n-1}$ \\
$(t,t+1,\ldots,t+n-1,-nt-\binom{n}{2})$ & $C_1\cdots C_{n-1}\cdot F(t,n)$ \\ 
$(n,n-2,n-4,\ldots,-n+4,-n+2,-n)$ & no closed formula \\
\hline
\end{tabular}
\end{center}
\caption{Summary of values of $K_n(\bm{N})$.}
\label{table: values kpf}
\end{table}

\subsubsection{The case of $2\rho$} \label{subsection: case rho}

We next consider the case where $\bm{N} = 2\rho(n)$ is the sum of all the positive roots of the type $A_n$ root system, given by $e_i - e_j$ for $0 \leq i < j \leq n$. The quantity $K_n(\bm{N})$ gives the dimension of the zero weight space of a certain \emph{Verma module} \cite{Humphreys}. More concretely, we have
\[
    2\rho := (n,n-2,n-4,\ldots,-n+4,-n+2,-n).
\]
A post in \cite{mathOverflow} raised the question of studying bounds for $K_n(2\rho)$. In an unpublished report \cite{JON2}, O'Neill using the techniques in \cite[\S 5]{JON} obtained the following bounds for $K_n(2\rho)$ and $K_n(t\cdot 2\rho)$.

\begin{proposition}[\cite{JON2}] \label{prop:bounds jason 2rho}
For a nonnegative odd integer $n=2k+1$ and $t>1$ we have that 
\[
K_{n}(2\rho) \geq 3^{k^2-k-1}, \qquad K_{n}(t\cdot 2 \rho) \geq \left(t+\frac12\right)^{k^2}.
\]
\end{proposition}

By taking the logs of the results above we immediately obtain the following results.

\begin{corollary} \label{cor:bounds Jason phi}
Fix an integer $t > 1$ and let $\bm{N}=2\rho(n)$. We have that
\[
    \log K_n(2\rho(n)) \geq \frac{1}{4} n^2 \log 3 - O(n),
\]
and 
\[
    \log K_n(t\cdot 2 \rho(n))\geq \frac{1}{4} (n-1)^2 \log\left(t+\frac{1}{2}\right).
\]
\end{corollary}

\subsection{Polynomial capacity and log-concave polynomials}

Polynomial capacity, originally defined by Gurvits \cite{G06}, is typically defined as follows. Note that here we extend the definition to multivariate power series.

\begin{definition}[Polynomial and power series capacity] \label{def:cap}
    Given a polynomial $p \in \R[\bm{x}] = \R[x_1,\ldots,x_n]$ or power series $p \in \R[[\bm{x}]] = \R[[x_1,\ldots,x_n]]$ with non-negative coefficients and any $\bm\alpha \in \R_{\geq 0}^n$, we define
    \[
        \cpc_{\bm\alpha}(p) = \inf_{\bm{x} > 0} \frac{p(\bm{x})}{\bm{x}^{\bm\alpha}} = \inf_{x_1,\ldots,x_n > 0} \frac{p(x_1,\ldots,x_n)}{x_1^{\alpha_1} \cdots x_n^{\alpha_n}}.
    \]
    This equivalently defined as
    \[
        \log\cpc_{\bm\alpha}(p) = \inf_{\bm{y} \in \R^n} \left[\log p(e^{\bm{y}}) - \langle \bm{y}, \bm\alpha \rangle\right],
    \]
    where $\langle \cdot, \cdot \rangle$ is the usual dot product.
\end{definition}

The typical use of capacity is to approximate or bound the coefficients of certain polynomials or power series. For example, the following bound follows immediately from the definition.

\begin{lemma}
    Given a polynomial or power series $p \in \R_{\geq 0}[\bm{x}] = \R_{\geq 0}[x_1,\ldots,x_n]$ with non-negative coefficients and any $\bm\alpha \in \Z_{\geq 0}^n$, we have
    \[
        \cpc_{\bm\alpha}(p) \geq [\bm{x}^{\bm\alpha}]\, p(\bm{x}),
    \]
    where $[\bm{x}^{\bm\alpha}]\, p(\bm{x})$ denotes the coefficient of $\bm{x}^{\bm\alpha}$ in $p$.
\end{lemma}

Lower bounds in terms of the capacity, on the other hand, are harder to prove. In fact, one should not expect such lower bounds in general; e.g., $[\bm{x}^{\bm{1}}]\, p(\bm{x}) = 0$ and $\cpc_{\bm{1}}(p) > 0$ for $p(\bm{x}) = x_1^n + \cdots + x_n^n$. Thus lower bounds are typically only proven for certain classes of polynomials and power series. The most common classes are real stable polynomials, Lorentzian polynomials (also known as completely log-concave and strongly log-concave \cite{BH,AGSVI,G09}), and most recently denormalized Lorentzian polynomials. Such bounds have been applied to various quantities, such as: the permanent, the mixed discriminant, and the mixed volume \cite{G06,G06MD,G09MV}; quantities related to matroids, like the number of bases of a matroid and the intersection of two matroids \cite{SV17,AO17,AGSVI,AGSVII}; the number of matchings of a bipartite graph \cite{GL21}; and the number of contingency tables \cite{Bar09,G15,BLP}. We do not explicitly use results of definitions regarding these polynomial classes, but instead refer the interested reader to the above references.

\subsection{Bounds on contingency tables}

Given vectors $\bm\alpha \in \Z_{\geq 0}^m$ and $\bm\beta \in \Z_{\geq 0}^n$, recall that a contingency table is an $m \times n$ matrix $M \in \Z_{\geq 0}^{m \times n}$ with non-negative integer entries, for which the row sums of $M$ and the column sums of $M$ are given by the entries of $\bm\alpha$ and $\bm\beta$ respectively. (Recall also that we index the rows and columns starting with $0$.) As discussed above, the set of all contingency tables has a nice generating function with coefficients indexed by $\bm\alpha$ and $\bm\beta$, given by
\[
    \Phi(\bm{x},\bm{y}) \,=\, \sum_{\bm\alpha,\bm\beta} \CT(\bm\alpha,\bm\beta) \, \bm{x}^{\bm\alpha} \bm{y}^{\bm\beta} \,=\, \prod_{i=0}^{m-1} \prod_{j=0}^{n-1} \frac{1}{1-x_iy_j}.
\]
Lower bounds on $\CT(\bm\alpha,\bm\beta)$ in terms of the capacity of $\Phi(\bm{x},\bm{y})$ were obtained in \cite{BLP}, which improved upon previous bounds of \cite{Bar09,G15}. The bound for general contingency tables is given as follows.

\begin{theorem}[\cite{BLP}, Thm. 2.1]
    Given $\bm\alpha \in \Z_{\geq 0}^m$ and $\bm\beta \in \Z_{\geq 0}^n$, we have
    \[
        \cpc_{\bm\alpha,\bm\beta}(\Phi) \,\geq\, \CT(\bm\alpha,\bm\beta) \,\geq\, \left[\prod_{i=1}^{m-1} \frac{\alpha_i^{\alpha_i}}{(\alpha_i+1)^{\alpha_i+1}} \prod_{j=0}^{n-1} \frac{\beta_j^{\beta_j}}{(\beta_j+1)^{\beta_j+1}}\right] \, \cpc_{\bm\alpha,\bm\beta}(\Phi).
    \]
    Note that $i,j$ starting from $1,0$ respectively is not a typo, and in fact we can replace the products above by a product over any subset of all but one of the factors.
\end{theorem}

Bounds are also achieved in \cite{BLP} for contingency tables with restricted entries, where then entries of a given matrix $M \in \Z_{\geq 0}^{m \times n}$ are bounded above entry-wise by a given matrix $K$. We care in this paper specifically about the case when $k_{ij} = +\infty$ for $i+j \leq n$ and $k_{ij} = 0$ otherwise, for which we have that the number of contingency tables counts the number of integer flows as in (\ref{eq: def injection flows to transportation matrix}). In this case, we define
\[
    \Phi'(\bm{x},\bm{y}) = \prod_{\substack{0 \leq i,j \leq n-1 \\ i+j \leq n}} \frac{1}{1-x_iy_j},
\]
and the bound is given as follows. Note that the authors of \cite{BLP} do not write their theorem in terms of flows explicitly, and so we translate their theorem here (in light of (\ref{eq: def injection flows to transportation matrix})) for the convenience of the reader.

\begin{theorem}[\cite{BLP}, Thm. 2.1] \label{thm:BLP-flow-bound}
    Given $\bm{N} = (N_0,N_1,\ldots,N_n) \in \Z^{n+1}$, define $s_i = \sum_{k=0}^i N_k$ for $i \in \{0,1,\ldots,n-1\}$ and let $\bm\alpha = (s_0,s_1,\ldots,s_{n-1})$ and $\bm\beta = (s_{n-1},\ldots,s_1,s_0)$. Then we have
    \[
        \cpc_{\bm\alpha,\bm\beta}(\Phi') \,\geq\, K_n(\bm{N}) \,\geq\, \max_{0 \leq i \leq n-1} \left\{\frac{(s_i+1)^{s_i+1}}{s_i^{s_i}}\right\} \, \left[\prod_{i=0}^{n-1} \frac{s_i^{s_i}}{(s_i+1)^{s_i+1}}\right]^2 \, \cpc_{\bm\alpha,\bm\beta}(\Phi').
    \]
\end{theorem}

In this paper, we will determine lower bounds for $\cpc_{\bm\alpha,\bm\beta}(\Phi')$ and combine this with the bound given by Theorem \ref{thm:BLP-flow-bound}. AThis will lead to our new bounds on flows.

\subsection{Convex analysis}

Given a function $f: \R^n \to (-\infty, +\infty]$, we define its domain $\mathcal{D} = \mathcal{D}(f)$ via
\[
    \mathcal{D} := \{\bm{x} \in \R^n : f(\bm{x}) < +\infty\}.
\]
We say $f$ is convex if $\mathcal{D}$ is convex and $f$ is convex on $\mathcal{D}$. The {\em convex conjugate} (or {\em Fenchel conjugate} or {\em Legendre transform}) of $f$ is also a convex function, defined as follows.

\begin{definition}
    Given a convex function $f$, its \textbf{convex conjugate} $f^*: \R^n \to (-\infty, +\infty]$ is a convex function defined via
    \[
        f^*(\bm{y}) := \sup_{\bm{x} \in \mathcal{D}} \left[\langle \bm{x}, \bm{y} \rangle - f(\bm{x})\right].
    \]
    We denote the domain of $f^*$ by $\mathcal{D}^* = \mathcal{D}(f^*)$.
\end{definition}

In order to give lower bounds on the capacity of a given polynomial or generating series, we use an idea already present in the work of Barvinok (e.g. Lemma 5 of \cite{Bar12}) and in Proposition 6.2 of \cite{BLP}: we convert the infimum in the definition of capacity (Definition \ref{def:cap}) into a supremum. We can then lower bound the supremum by simply evaluating the objective function at any particular chosen value of the domain. A proof sketch for a general version of this is given in \cite{BLP}, and we give a different and simpler proof in this paper based on the following classical result of convex analysis relating convex conjugates via the {\em infimal convolution}.

\begin{theorem}[\cite{rockafellar1970convex}, Thm. 16.4] \label{thm:conjugate_split} 
    Let $f$ be a convex function given as the sum $f = \sum_i f_i$ of convex functions with respective domains $\mathcal{D}_i$. If $\bigcap_i \relint(\mathcal{D}_i)$ is non-empty then
    \[
        f^*(\bm\alpha) = \inf_{\sum_i \bm\alpha_i = \bm\alpha} \sum_i f_i^*(\bm\alpha_i),
    \]
    where for each $\bm\alpha$ the infimum is attained. Note that the domain of optimization is over all choices of vectors $\bm\alpha_i \in \mathcal{D}_i$ such that $\sum_i \bm\alpha_i = \bm\alpha$.
\end{theorem}

Note that Theorem \ref{thm:conjugate_split} converts a supremum into an infimum, rather than the other way around. Since the $\log$ of capacity is the negation of a convex conjugate, Theorem \ref{thm:conjugate_split} then gives precisely what we need.

\section{A Dual Formulation for Capacity} \label{sec:dual-capacity}

As above, let $K_n(\bm{N})$ denote the number of integer flows on $k_{n+1}$ with netflow given by $\bm{N} \in \Z^{n+1}$. Counting such integer flows is equivalent to counting the integer matrices of the form given by (\ref{eq: def injection flows to transportation matrix}), where the row and columns sums are given by $\bm\alpha = (s_0,s_1,\ldots,s_{n-1})$ and $\bm\beta = (s_{n-1},\ldots,s_1,s_0)$ where $s_k = \sum_{j=0}^k N_j$. Thus by Theorem \ref{thm:BLP-flow-bound}, we have
\[
    \cpc_{\bm\alpha,\bm\beta}(\Phi') \,\geq\, K_n(\bm{N}) \,\geq\, \max_{0 \leq i \leq n-1} \left\{\frac{(s_i+1)^{s_i+1}}{s_i^{s_i}}\right\} \, \left[\prod_{i=0}^{n-1} \frac{s_i^{s_i}}{(s_i+1)^{s_i+1}}\right]^2 \, \cpc_{\bm\alpha,\bm\beta}(\Phi'),
\]
where
\[
    \Phi'(\bm{x},\bm{y}) = \prod_{i+j \leq n} \frac{1}{1-x_iy_j}
\]
with $0 \leq i,j \leq n-1$.

In this section, we will utilize Theorem \ref{thm:conjugate_split}  to convert the infimum of the above capacity expression into a supremum. Specifically, we will prove the following.

\begin{proposition} \label{prop:cap_sup}
    Let $\phi(\mathcal{F}_n(\bm{N}))$ be the image of $\mathcal{F}_n(\bm{N})$ in $\mathcal{T}(\bm\alpha,\bm\beta)$ as defined in (\ref{eq: def injection flows to transportation matrix}), and recall the definition of flow entropy $\mathcal{H}(\bm{f})$ from (\ref{eq:flow-entropy}). We have that
    \[
        \cpc_{\bm\alpha \, \bm\beta}\left(\prod_{i+j \leq n} \frac{1}{1-x_iy_j}\right) \,=\, \sup_{A \in \phi(\mathcal{F}_n(\bm{N}))} \prod_{i+j \leq n} \frac{(a_{ij}+1)^{a_{ij}+1}}{a_{ij}^{a_{ij}}} \,=\, \sup_{\bm{f} \in \mathcal{F}_n(\bm{N})} e^{\mathcal{H}(\bm{f})}.
    \]
\end{proposition}

Proposition \ref{prop:cap_sup} then leads immediately to the following result, which we will utilize in the later sections.

\begin{theorem} \label{thm:matrix_lower_bound}
    Let $\bm{f}$ be any (not necessarily integer) point of $\mathcal{F}_n(\bm{N})$, let $s_k = \sum_{j=0}^k N_j$, and let $A = \phi(\bm{f})$ where $\phi$ is defined as in (\ref{eq: def injection flows to transportation matrix}).
    We have that
    \[
        K_n({\bm{N}}) \,\geq\, \max_{0 \leq i \leq n-1} \left\{\frac{(s_i+1)^{s_i+1}}{s_i^{s_i}}\right\} \, \left[\prod_{i=0}^{n-1} \frac{s_i^{s_i}}{(s_i+1)^{s_i+1}}\right]^2 \, \prod_{i+j \leq n} \frac{(a_{ij}+1)^{a_{ij}+1}}{a_{ij}^{a_{ij}}}.
    \]
\end{theorem}

There are also versions of the above results for the volumes of flow polytopes. Since these are outside the context of the results of this paper, we leave further discussion of these results to the final remarks (see Section \ref{sec:approx-vol-flows}).

\subsection{Proof of Proposition \ref{prop:cap_sup}} \label{sec:proof_cap_sup}

The second equality follows from the definition of flow entropy $\mathcal{H}(\bm{f})$, and so we just need to prove the first equality. Consider the following function:
\[
    f_{\#}(\bm{x},\bm{y}) := -\sum_{i+j \leq n} \log\left(1 - e^{x_i+y_j}\right),
\]
Here, as above, the variables are indexed from $0$ to $n-1$. Since $-\log(1-e^t)$
is a convex function on its domain, we have that $f_{\#}$
is convex on its domain $\mathcal{D}_{\#} \supseteq \R_{<0}^{2n}$.
Since $f_{\#}$ is defined as a sum of convex functions, we can apply Theorem \ref{thm:conjugate_split} to obtain
\[
    f_{\#}^*(\bm\alpha,\bm\beta) = \inf_{\sum_{i+j \leq n} (\bm\alpha_{i,j},\bm\beta_{i,j}) = (\bm\alpha,\bm\beta)} \sum_{i+j \leq n} f_{\#;i,j}^*(\bm\alpha_{i,j},\bm\beta_{i,j})
\]
for any $\bm\alpha,\bm\beta$ of the form described at the start of Section \ref{sec:dual-capacity}. Note that the sum under the $\inf$ is over all choices of vectors $\bm\alpha_{i,j}$ and $\bm\beta_{i,j}$ (for $0 \leq i,j \leq n-1$ and $i+j \leq n$) such that $\sum_{i+j \leq n} (\bm\alpha_{i,j}, \bm\beta_{i,j}) = (\bm\alpha,\bm\beta)$.

Here we have that
\[
    f_{\#;i,j}(\bm{x},\bm{y}) := -\log\left(1-e^{x_i+y_j}\right),
\]
and the function $f_{\#;i,j}$ is convex with domain given by
\[
    \mathcal{D}_{\#;i,j} := \{(\bm{x},\bm{y}) \in \R^{2n} : x_i+y_j < 0\}.
\]
We then have
\[
    f_{\#;i,j}^*(\bm\alpha,\bm\beta) = \sup_{(\bm{x},\bm{y}) \in \mathcal{D}_{\#;i,j}} \left[\langle \bm{x}, \bm\alpha \rangle + \langle \bm{y}, \bm\beta \rangle + \log\left(1 - e^{x_i+y_j}\right)\right],
\]
and by a straightforward argument this implies
\[
    \mathcal{D}_{\#;i,j}^* = \{c \cdot (\bm{e}_i, \bm{e}_j) : c \in [0,\infty)\}.
\]
Thus for any $c \geq 0$, standard calculus arguments give
\[
    f_{\#;i,j}^*(c \bm{e}_i, c \bm{e}_j) = \sup_{x_i,y_j} \left[c(x_i + y_j) + \log\left(1-e^{x_i+y_j}\right)\right] = \sup_{t < 0} \left[ct + \log\left(1-e^t\right)\right] = \log\left(\frac{c^c}{(c+1)^{c+1}}\right).
\]
Combining this with the above expressions then gives
\[
    f_{\#}^*(\bm\alpha,\bm\beta) = \inf_{A \in \phi(\mathcal{F}_n(\bm{N}))} \sum_{i+j \leq n} \log\left(\frac{a_{ij}^{a_{ij}}}{(a_{ij}+1)^{a_{ij}+1}}\right).
\]
Negating and exponentiating both sides then gives the desired result.

\section{Computing the Average of Vertices of some flow polytopes} \label{sec:avg-of-vertices}

In this section we compute the uniform average of the vertices of the flow polytopes $\mathcal{F}_n(\bm{N})$ for $\bm{N}$ with positive netflow $N_i>0$ and $\bm{N}=(1,0,\ldots,0,-1)$.  We assume a uniform distribution on the vertices of the polytope. In abuse of notation we refer to all the entries on or above the antidiagonal, upper triangular entries (see Proposition~\ref{prop: flow is face of trans poly}).

\begin{proposition} \label{prop: center of mass positive netflow}
For $\bm{N}=(N_0,\ldots,N_{n-1},-\sum_{i=0}^{n-1} N_i)$ where $N_i>0$ the uniform average of the vertices of $\mathcal{F}_n(\bm{N})$ is the flow $f(i,j)=c_{n-i}$ where
$c_k = \dfrac{s_{n-k}}{k} - \dfrac{s_{n-k-1}}{k+1} = \dfrac{N_{n-k}}{k+1} + \dfrac{s_{n-k}}{k(k+1)}$, for $s_k=\sum_{j=0}^k N_j$. That is, it is represented by the matrix
\begin{equation} \label{eq:center of mass Tesler case}
    A = \begin{bmatrix}
        c_n & c_n & c_n & \cdots & c_n & c_n & c_n \\
        c_{n-1} & c_{n-1} & c_{n-1} & \cdots & c_{n-1} & c_{n-1} & b_{n-1} \\
        c_{n-2} & c_{n-2} & c_{n-2} & \cdots & c_{n-2} & b_{n-2} & 0 \\
        \vdots & \vdots & \vdots & \ddots & \vdots & \vdots & \vdots \\
        c_3 & c_3 & c_3 & \cdots & 0 & 0 & 0 \\
        c_2 & c_2 & b_2 & \cdots & 0 & 0 & 0 \\
        c_1 & b_1 & 0 & \cdots & 0 & 0 & 0 \\
    \end{bmatrix},
\end{equation}
where $b_k = \frac{k}{k+1} s_{n-k-1}$.
\end{proposition}

\begin{proof}

Let $A=(a_{i,j})$ be desired uniform average, and let $X=(x_{i,j})$ be a uniformly random vertex. We first show that $a_{i,j} = c_{n-i}$ when $j \leq n-i$ by induction on $i$. Note that by Theorem~\ref{lemma:char vertices all ones}, each upper triangular entry of row $i$ of a vertex is nonzero in exactly $\frac{1}{n-i}$ fraction of the vertices. Then using this fact and \eqref{eq: vert} 
\[
    a_{i,j} = \mathbb{E}[x_{i,j}] = \frac{1}{n-i}\left(N_i + \sum_{r=0}^{i-1} \mathbb{E}[x_{r,n-i}]\right).
\]
By induction we compute
\[
\begin{split}
    a_{i,j} &= \frac{1}{n-i}\left(N_i + \sum_{r=0}^{i-1} c_{n-r}\right)  = \frac{1}{n-i}\left(N_i + \sum_{r=0}^{i-1} \left(\frac{s_r}{n-r} - \frac{s_{r-1}}{n-(r-1)}\right)\right) \\
    &= \frac{N_i}{n-i} + \frac{s_{i-1}}{(n-i)(n-i+1)} =
        \frac{N_i}{n-i} + \frac{s_i}{(n-i)(n-i+1)} - \frac{N_i}{(n-i)(n-i+1)} = c_{n-i}.
\end{split}
\]
\end{proof}

\begin{proposition} \label{prop: center of mass CRY netflow}
For $\bm{N}=(t,0,\ldots,0,-t)$ where $t>0$ the uniform average of vertices of $\mathcal{F}_n(\bm{N})$ is represented by the matrix $t\cdot A$ where 
 \begin{equation} \label{eq:center of mass CRY case}
        A := \begin{bmatrix}
            2^{-(n-1)} & 2^{-(n-1)} & 2^{-(n-2)} & \cdots & 2^{-3} & 2^{-2} & 2^{-1} \\
            2^{-(n-1)} & 2^{-(n-1)} & 2^{-(n-2)} & \cdots & 2^{-3} & 2^{-2} & 2^{-1} \\
            2^{-(n-2)} & 2^{-(n-2)} & 2^{-(n-3)} & \cdots & 2^{-2} & 2^{-1} & 0 \\
            \vdots & \vdots & \vdots & \ddots & \vdots & \vdots & \vdots \\
            2^{-3} & 2^{-3} & 2^{-2} & \cdots & 0 & 0 & 0 \\
            2^{-2} & 2^{-2} & 2^{-1} & \cdots & 0 & 0 & 0 \\
            2^{-1} & 2^{-1} & 0 & \cdots & 0 & 0 & 0
        \end{bmatrix}.
    \end{equation}
\end{proposition}

\begin{proof}
    Let $A=(a_{i,j})$ be desired uniform average for $t=1$, and let $X=(x_{i,j})$ be a uniformly random vertex. By Theorem~\ref{lemma:char vertices CRY}, $x_{i,n-i} = v_i$ are i.i.d. uniform Bernoulli random variables for all $1 \leq i \leq n-1$, and $v_0 = v_n = 0$. Using this and the description of the vertices of $\mathcal{F}_n({\bm{N}})$ in \eqref{eq:vertices cry case}, if $i,j \geq 1$ then
    \[
        a_{i,j} = \mathbb{E}[x_{i,j}] = \begin{cases}
            (1-\mathbb{E}[v_i])(1-\mathbb{E}[v_{n-j}])\prod_{k=i+1}^{n-j-1} \mathbb{E}[v_k] = 2^{-(n-i-j+1)} & \text{if } i<n-j \\
            \mathbb{E}[v_i] = 2^{-1} & \text{if } i=n-j \\
            0 &\text{otherwise}.
            \end{cases}
    \]
    Further, if exactly one of $i,j$ is equal to 0 then
    \[
        a_{i,j} = \mathbb{E}[x_{i,j}] = (1-\mathbb{E}[v_i])(1-\mathbb{E}[v_{n-j}])\prod_{k=i+1}^{n-j-1} \mathbb{E}[v_k] = 2^{-(n-i-j)}.
    \]
    And finally, if $i=j=0$ then
    \[
        a_{i,j} = \mathbb{E}[x_{i,j}] = (1-\mathbb{E}[v_i])(1-\mathbb{E}[v_{n-j}])\prod_{k=i+1}^{n-j-1} \mathbb{E}[v_k] = 2^{-(n-i-j-1)} = 2^{-(n-1)}.
    \]
    For the case $\bm{N}=(t,0,\ldots,0,-t)$ for $t>0$ we have that $\mathcal{F}_n(\bm{N})=t\cdot \mathcal{F}(1,0,\ldots,0,-1)$, and so the uniform average of the vertices also dilates by $t$.
 \end{proof}

For the case that $\bm{N} = t \cdot 2\rho(n)$ (where $2\rho(n) = (n,n-2,n-4,\ldots,-n+4,-n+2,-n)$, see Section~\ref{subsection: case rho}) with $s_k = t (k+1)(n-k)$ for all $k$, we were unable to exactly compute the average of the vertices. However, a few experiments show that the average may be close to the following natural point in the flow polytope:
\begin{equation} \label{eq:2rho-midpoint}
    M := t \cdot \begin{bmatrix}
        1 & 1 & 1 & \cdots & 1 & 1 & 1 \\
        1 & 1 & 1 & \cdots & 1 & 1 & n-1 \\
        1 & 1 & 1 & \cdots & 1 & 2(n-2) & 0 \\
        \vdots & \vdots & \vdots & \ddots & \vdots & \vdots & \vdots \\
        1 & 1 & 1 & \cdots & 0 & 0 & 0 \\
        1 & 1 & 2(n-2) & \cdots & 0 & 0 & 0 \\
        1 & n-1 & 0 & \cdots & 0 & 0 & 0
    \end{bmatrix},
\end{equation}
where the subdiagonal entries are given by $k(n-k)$ for $1 \leq k \leq n-1$.

\begin{remark}
As stated above, we were not able to compute the average of the vertices of $\mathcal{F}_n(\bm{N})$ where $\bm{N}=2\rho(n)$. The polytope for the cases $n=2,\ldots,5$ have $2,7,26$ and $219$ vertices and averages:
\[
{\footnotesize 
\begin{bmatrix}
1& 1 \\
1 & 1
\end{bmatrix}
,\,
\begin{bmatrix}
\frac{5}{7} & \frac{8}{7} & \frac{8}{7} \\
\frac{8}{7} & 1 & \frac{13}{8} \\
\frac{8}{7} & \frac{13}{8} & 0
\end{bmatrix}
,\,
\begin{bmatrix}
\frac{7}{13} & \frac{14}{13} & \frac{16}{13} & \frac{15}{13} \\
\frac{14}{13} & \frac{9}{13} & \frac{18}{13} & \frac{37}{13} \\
\frac{16}{13} & \frac{18}{13} & \frac{54}{13} & 0 \\
\frac{15}{13} & \frac{37}{13}  & 0 & 0
\end{bmatrix}
,\,
 \begin{bmatrix}
0.553 & 1.078 & 1.132 & 1.105 & 1.132 \\
1.078 & 0.680 & 1.187 & 1.187 &  3.868 \\
1.132 & 1.187 & 0.973 & 5.708 & 0 \\
1.105 & 1.187 & 5.708 & 0 & 0 \\
1.132 & 3.868 & 0 & 0 & 0
\end{bmatrix}
.
}
\]

It would be interesting to find the number of vertices and average for this case.

\end{remark}

\section{Flow Counting Lower Bounds} \label{sec:flow-counting-lower-bounds}

In this section we prove our main lower bounds on flows. Specifically we apply Theorem \ref{thm:matrix_lower_bound} to various flow vectors $\bm{N}$, using specific choices of flows $\bm{f}$ given by the average of the vertices of the associated flow polytopes (as computed in Section \ref{sec:avg-of-vertices}). This technique yields lower bounds for the number of flows $K_n(\bm{N})$, often given in a relatively complicated product form. We then obtain more explicit lower bounds for the asymptotics of $\log K_n(\bm{N})$ by combining the Euler-Maclaurin formula (see Lemma \ref{lem:euler-maclaurin}) with a number of elementary bounds on entropy-like functions (see Lemma \ref{lem:basic-entropy-bound} and the rest of Appendix \ref{sec:bounds_asymptotics}).

We now state the bounds obtained in the section in the following results, and the remainder of this section is devoted to proving these bounds. Throughout, as above, we let $\bm{N} = (N_0,N_1,\ldots,N_n) \in \Z^{n+1}$ denote the netflow vector, and we denote $s_k = \sum_{j=0}^k N_j$. Any big-O notation used is always with respect to $n$, with other parameters fixed. Also, we will sometimes put parameters in the subscripts of the big-O notation to denote that that implied constant may depend on those parameters.

\begin{theorem}[Polynomial growth] \label{thm:general_positive_lower_bound}
    For $N_k \geq a \cdot k^p$ for all $k \in \{0,1,\ldots,n-1\}$, with given $a > 0$ and $p \geq 0$,
    \[
        \log K_n({\bm{N}}) \geq
        \begin{cases}
       
            n^2 \log n \cdot \left(\frac{p-1}{2}\right) + \frac{n^2}{2} \cdot \left(\log(ap) - \frac{3(p-1)}{2}\right) - O\left(n^{1+\frac{1}{p}}\right) & p > 1 \\
         
            n^2 \cdot \left(\frac{a}{2(a-2)} \log\left(\frac{a}{2}\right) + \frac{3}{2} - 2\log 2\right) - O(n \log n), & p = 1, a > 2 \\
            n^2 \cdot \left(a - a \log 2)\right) - O(n \log n), & p = 1, a \leq 2 \\
          
            n^{p+1} \log^2 n \cdot \left(\frac{a(1-p)^2}{4(p+1)}\right) - O(n^{p+1} \log n \log\log n) & p < 1
        \end{cases}.
    \]
    Note that the implied constant of the big-O notation may depend on $a$ and $p$. Also note that the $p=1$ cases limit to the same bounds at $a=2$.
\end{theorem}

\begin{theorem}[Tesler case] \label{thm:in-section-tesler}
    For $\bm{N} = (1,1,\ldots,1,-n)$,
    \[
        \log K_n(\bm{N}) \geq \frac{n}{4} \log^2 n - O(n \log n).
    \]
    Further, $\log K_n(\bm{N}) \geq (n-1) \log(n+1)$ for $n \geq 3000$.
\end{theorem}

\begin{theorem}[Other specific examples] \label{thm:other-examples}
    We have the following.
    \begin{enumerate}
        \item For $\bm{N} = (n,n,\ldots,n,-n^2)$,
        \[
            \log K_n(\bm{N}) \geq n^2 - O(n \log n).
        \]
        \item For $N_i = a \cdot n$ for all $i \in \{0, 1, \ldots, n-1\}$, with given $a \geq \frac{1}{12}$,
        \[
            \log K_n(\bm{N}) \geq \frac{n^2}{2} (2 + \log a) - O(n \log n).
        \]
        \item For $N_i = a \cdot n + i$ for all $i \in \{0, 1, \ldots, n-1\}$, with given $a \geq 0$,
        \[
            \log K_n(\bm{N}) \geq \frac{n^2}{2} \left(1 + \log(2a + 1)\right) - O(n \log n).
        \]
        Note that the implied constant may depend on $a$.
        \item For $N_i = n + i$ for all $i \in \{0, 1, \ldots, n-1\}$,
        \[
            \log K_n(\bm{N}) \geq 1.198 n^2 - O(n \log n).
        \]
        \item For $\bm{N} = (t,0,0,\ldots,0,-t)$, with given $t \geq 1$,
        \[
            \log K_n(\bm{N}) \geq \frac{n}{2} \log_2^2 t - O(n \log_2 t).
        \]
        The implied constant is independent of $t$.
        \item For $\bm{N} = t \cdot 2\rho(n) = t \cdot (n, n-2, n-4, \ldots, -n+2, -n)$, with given $t \geq 1$,
        \[
            \log K_n(\bm{N}) \geq \frac{n^2}{2} \log \left(\frac{(1+t)^{1+t}}{t^t}\right) - O(n \log(nt)).
        \]
        The implied constant is independent of $t$.
    \end{enumerate}
\end{theorem}

\begin{remark}
   
    Using Theorem \ref{thm:case M_n(t,0,0)} and Proposition \ref{prop:F-f-bound}, for $N_i = a \cdot n + i$ we have
    \[
        \log K_n(\bm{N}) = \log F(a \cdot n, n) + \sum_{k=1}^{n-1} \log C_k,
    \]
    where $C_k = \frac{1}{k+1} \binom{2k}{k}$ is the $k$th Catalan number and
    \[
        -2(a+1)n \leq \log F(a \cdot n, n) - n^2 (a+1)^2 f\left(\frac{a}{a+1}\right) \leq 0
    \]
    with
    \[
        f(x) = x^2 \log x - \frac{1}{2} (1-x)^2 \log (1-x) - \frac{1}{2} (1+x)^2 \log (1+x) + 2x \log 2.
    \]
    For large $a$, standard computations give
    \[
        (a+1)^2 f\left(\frac{a}{a+1}\right) = \frac{1}{2} \log a \pm O(1),
    \]
    and
    \[  
        \sum_{k=1}^{n-1} \log C_k = n^2 \log 2 - O(n \log n).
    \]
    This demonstrates that our bound above for $N_i = a \cdot n + i$ achieves the correct leading term in $n$ and in $a$.
\end{remark}

\subsection{Positive flows in general}

Here we state some general lower bounds in the cases where every entry of the netflow vector is positive. Note though that these bounds hold even in the case where the netflow vector is only non-negative.

\begin{theorem} \label{thm:general_lower_bound}
    Fix $\bm{N} \in \Z^{n+1}$ (i.e., the entries are not necessarily non-negative). Denoting $s_k = \sum_{j=0}^k N_j$ and $c_k = \frac{N_{n-k}}{k+1} + \frac{s_{n-k}}{k(k+1)}$, we have
    \begin{enumerate}
        \item If $s_k \geq \max\{0,-(n-k)N_k\}$ for all $k \in \{0,1,\ldots,n-1\}$,
        \[
            K_n({\bm{N}}) \geq \frac{1}{n^2} \prod_{k=0}^{n-1} \frac{s_k^{s_k}}{(1+s_k)^{1+s_k}} \prod_{k=1}^n \left(\frac{(c_k+1)^{c_k+1}}{c_k^{c_k}}\right)^k.
        \]
        \item If $s_k \geq \max\{0,-(n-k)N_k\}$ and $N_k \geq \frac{1}{n-k} - (n-k+1)$ for all $k \in \{0,1,\ldots,n-1\}$,
        \[
            K_n({\bm{N}}) \geq \frac{1}{n^2 e^n (n!)^2} \prod_{k=1}^n \left(\frac{(c_k+1)^{c_k+1}}{c_k^{c_k}}\right)^{k-1}.
        \]
    \end{enumerate}
\end{theorem}
\begin{proof}
    As in Proposition \ref{prop: center of mass positive netflow}, we define
    \[
        A = \begin{bmatrix}
            c_n & c_n & c_n & \cdots & c_n & c_n & c_n \\
            c_{n-1} & c_{n-1} & c_{n-1} & \cdots & c_{n-1} & c_{n-1} & b_{n-1} \\
            c_{n-2} & c_{n-2} & c_{n-2} & \cdots & c_{n-2} & b_{n-2} & 0 \\
            \vdots & \vdots & \vdots & \ddots & \vdots & \vdots & \vdots \\
            c_3 & c_3 & c_3 & \cdots & 0 & 0 & 0 \\
            c_2 & c_2 & b_2 & \cdots & 0 & 0 & 0 \\
            c_1 & b_1 & 0 & \cdots & 0 & 0 & 0 \\
        \end{bmatrix},
    \]
    where $c_k = \frac{s_{n-k}}{k} - \frac{s_{n-k-1}}{k+1} = \frac{N_{n-k}}{k+1} + \frac{s_{n-k}}{k(k+1)}$ and $b_k = \frac{k}{k+1} s_{n-k-1}$.
    
    We first prove part (1) of Theorem \ref{thm:general_lower_bound}. Using Theorem \ref{thm:matrix_lower_bound}, we have
    \[
        K_n({\bm{N}}) \geq \max_k\left\{\frac{(1+s_k)^{1+s_k}}{s_k^{s_k}}\right\} \prod_{k=0}^{n-1} \left(\frac{s_k^{s_k}}{(1+s_k)^{1+s_k}}\right)^2 \prod_{k=1}^{n-1} \frac{(b_k+1)^{b_k+1}}{b_k^{b_k}} \prod_{k=1}^n \left(\frac{(c_k+1)^{c_k+1}}{c_k^{c_k}}\right)^k.
    \]
    Further note that for all $k$ we have
    \[
    \begin{split}
        \frac{(b_k+1)^{b_k+1}}{b_k^{b_k}} &= \frac{k}{k+1} \cdot \frac{(s_{n-k-1} + \frac{k+1}{k})^{\frac{k}{k+1} s_{n-k-1}+1}}{(s_{n-k-1})^{\frac{k}{k+1} s_{n-k-1}}} \\
            &\geq \frac{k}{k+1} \cdot \left(\frac{s_{n-k-1}}{s_{n-k-1}+1}\right)^{\frac{s_{n-k-1}}{k+1}} \cdot \frac{(s_{n-k-1} + 1)^{s_{n-k-1}+1}}{(s_{n-k-1})^{s_{n-k-1}}} \\
            &\geq \frac{k}{k+1} \cdot \left(\frac{1}{e}\right)^{\frac{1}{k+1}} \cdot \frac{(s_{n-k-1} + 1)^{s_{n-k-1}+1}}{(s_{n-k-1})^{s_{n-k-1}}},
    \end{split}
    \]
    since $\left(\frac{x}{x+1}\right)^x \geq \frac{1}{e}$ for all $x > 0$. Combining the above expressions then gives
    \[
    \begin{split}
        K_n({\bm{N}}) &\geq \prod_{k=0}^{n-1} \frac{s_k^{s_k}}{(1+s_k)^{1+s_k}} \prod_{k=1}^{n-1} \left[\frac{k}{k+1} \left(\frac{1}{e}\right)^{\frac{1}{k+1}}\right] \prod_{k=1}^n \left(\frac{(c_k+1)^{c_k+1}}{c_k^{c_k}}\right)^k \\
            &\geq \frac{1}{n^2} \prod_{k=0}^{n-1} \frac{s_k^{s_k}}{(1+s_k)^{1+s_k}} \prod_{k=1}^n \left(\frac{(c_k+1)^{c_k+1}}{c_k^{c_k}}\right)^k,
    \end{split}
    \]
    using the fact that $\sum_{k=2}^n \frac{1}{k} \leq \int_1^n \frac{1}{x} dx = \log n$. This implies part (1) of Theorem \ref{thm:general_lower_bound}.
    
    We next prove part (2) of Theorem \ref{thm:general_lower_bound}, which is mainly useful in the case that $N_k \geq 0 \geq \frac{1}{n-k} - (n-k+1)$ for all $k = 0,1,\ldots,n-1$. We first have
    \[
        \frac{(c_k+1)^{c_k+1}}{c_k^{c_k}} \geq c_k+1 = \frac{N_{n-k}}{k+1} + \frac{s_{n-k}}{k(k+1)} + 1 \geq \frac{1}{k(k+1)}(1 + s_{n-k})
    \]
    for $k = 1,2,\ldots,n$, since we have assumed that $N_{n-k} \geq \frac{1}{k} - (k+1)$ in part (3). In addition, when $k=n$ we have
    \[
        \frac{(c_n+1)^{c_n+1}}{c_n^{c_n}} \geq c_n + 1 = 1 + \frac{N_0}{n+1} + \frac{s_0}{n(n+1)} = 1 + \frac{s_0}{n} \geq \frac{1}{n}(1+s_0).
    \]
    Since $\left(\frac{x}{1+x}\right)^x \geq \frac{1}{e}$ for all $x > 0$, we then further have
    \[
        \frac{s_k^{s_k}}{(1+s_k)^{1+s_k}} \geq \frac{1}{e(1+s_k)}.
    \]
    Applying part (1) of Theorem \ref{thm:general_lower_bound} then gives
    \[
    \begin{split}
        K_n({\bm{N}}) &\geq \frac{1}{n^2} \prod_{k=0}^{n-1} \frac{s_k^{s_k}}{(1+s_k)^{1+s_k}} \prod_{k=1}^n \left(\frac{(c_k+1)^{c_k+1}}{c_k^{c_k}}\right)^k \\
            &\geq \frac{1}{n^2 e^n} \cdot \frac{1}{n} \prod_{k=1}^{n-1} \frac{1}{k(k+1)} \prod_{k=1}^n \left(\frac{(c_k+1)^{c_k+1}}{c_k^{c_k}}\right)^{k-1} \\
            &\geq \frac{1}{n^2 e^n (n!)^2} \prod_{k=1}^n \left(\frac{(c_k+1)^{c_k+1}}{c_k^{c_k}}\right)^{k-1}.
    \end{split}
    \]
    This implies part (2) of Theorem \ref{thm:general_lower_bound}.
\end{proof}

\subsection{Polynomial growth}

As in Proposition \ref{prop: center of mass positive netflow}, we define
\[
    A = \begin{bmatrix}
        c_n & c_n & c_n & \cdots & c_n & c_n & c_n \\
        c_{n-1} & c_{n-1} & c_{n-1} & \cdots & c_{n-1} & c_{n-1} & b_{n-1} \\
        c_{n-2} & c_{n-2} & c_{n-2} & \cdots & c_{n-2} & b_{n-2} & 0 \\
        \vdots & \vdots & \vdots & \ddots & \vdots & \vdots & \vdots \\
        c_3 & c_3 & c_3 & \cdots & 0 & 0 & 0 \\
        c_2 & c_2 & b_2 & \cdots & 0 & 0 & 0 \\
        c_1 & b_1 & 0 & \cdots & 0 & 0 & 0 \\
    \end{bmatrix},
\]
where $c_k = \frac{s_{n-k}}{k} - \frac{s_{n-k-1}}{k+1} = \frac{N_{n-k}}{k+1} + \frac{s_{n-k}}{k(k+1)}$ and $b_k = \frac{k}{k+1} s_{n-k-1}$. Since $N_k \geq a \cdot k^p$ for all $k \in \{0,1,\ldots,n-1\}$, (\ref{eq:sum_p}) implies
\[
    c_k = \frac{N_{n-k}}{k+1} + \frac{s_{n-k}}{k(k+1)} \geq \frac{a(n-k)^p}{k+1} + \frac{a(n-k)^{p+1}}{k(k+1)(p+1)} = \frac{a(n + kp)(n-k)^p}{k(k+1)(p+1)}.  
\]

\subsubsection{The case of $p > 1$.}

Define $a=1$ and $b = \lfloor n - \left(\frac{n}{ae}\right)^{1/p} \rfloor$. Using Lemma \ref{lem:basic-entropy-bound}, we first have have
\[
    \log \prod_{k=1}^n \left(\frac{(c_k+1)^{c_k+1}}{c_k^{c_k}}\right)^{k-1} \geq \sum_{k=a}^b (k-1) \log (e \cdot c_k).
\]
Using Lemma \ref{lem:klogk}, we have
\[
    \sum_{k=a}^b k \log(n+kp) = \frac{n^2}{2} \left[\log n + \left(1-\frac{1}{p^2}\right)\log(p+1) + \log p + \frac{1}{p} - \frac{1}{2}\right] - \frac{n^{1+\frac{1}{p}} \log n}{(ae)^{\frac{1}{p}}} \pm O_{a,p}\left(n^{1+\frac{1}{p}}\right),
\]
and
\[
    \sum_{k=a}^b pk \log(n-k) = \frac{n^2}{2}\left[p\log n - \frac{3p}{2}\right] - \frac{n^{1+\frac{1}{p}} \log n}{(ae)^{\frac{1}{p}}} \pm O_{a,p}\left(n^{1+\frac{1}{p}}\right),
\]
and
\[
    \sum_{k=a}^b k \log k = \sum_{k=a}^b k \log(k+1) = \frac{n^2}{2}\left[\log n - \frac{1}{2}\right] - \frac{n^{1+\frac{1}{p}} \log n}{(ae)^{\frac{1}{p}}} \pm O_{a,p}\left(n^{1+\frac{1}{p}}\right),
\]
and
\[
    \sum_{k=a}^b k \log\left(\frac{ae}{p+1}\right) = \frac{n^2}{2} \log\left(\frac{ae}{p+1}\right) \pm O_{a,p}\left(n^{1+\frac{1}{p}}\right).
\]
Combining everything gives
\[
    \sum_{k=a}^b k \log \frac{ae(n + kp)(n-k)^p}{k(k+1)(p+1)} = \frac{n^2}{2} \left[(p-1)\left(\log n - \frac{3}{2}\right) + \log a + \log p + \frac{p - \log(p+1)}{p^2}\right] - O_{a,p}\left(n^{1+\frac{1}{p}}\right).
\]
Note that $p \geq \log(p+1)$. Further,
\[
    \sum_{k=a}^b \log \frac{ae(n + kp)(n-k)^p}{k(k+1)(p+1)} \leq \sum_{k=1}^n \log \left(aen^{1+p}\right) = O_{a,p}(n \log n).
\]
By Theorem \ref{thm:general_lower_bound} (2), we then have
\[
    \log K_n(\bm{N}) \geq \log \prod_{k=1}^n \left(\frac{(c_k+1)^{c_k+1}}{c_k^{c_k}}\right)^{k-1} - O(n \log n)
\]
which implies the desired result.

\subsubsection{The case of $p=1$.}

First,
\[
    c_k \geq \frac{a(n^2-k^2)}{2k(k+1)} =: c_k'
\]
is a lower bound on the possible values of $c_k$ which we will use throughout the $p=1$ case. Defining
\[
    S = \sum_{k=1}^n k \log(1+c_k') + \sum_{k=1}^n k c_k' \log\left(1 + (c_k')^{-1}\right),
\]
we have
\[
    \log \prod_{k=1}^n \left(\frac{(c_k+1)^{c_k+1}}{c_k^{c_k}}\right)^{k-1} \geq S - \sum_{k=1}^n \log(1+c_k') - \sum_{k=1}^n c_k' \log\left(1 + (c_k')^{-1}\right)
\]
since $\frac{(c_k+1)^{c_k+1}}{c_k^{c_k}}$ is increasing in $c_k$. Further note that since $\log(1+t) \leq t$, we have
\[
    \sum_{k=1}^n \log\left(1+c_k'\right) + \sum_{k=1}^n c_k' \log\left(1 + (c_k')^{-1}\right) \leq n \log((a+1)n^2) + n = O(n \log n).
\]
Theorem \ref{thm:general_lower_bound} (2) then implies
\[
    \log K_n(\bm{N}) \geq S - O(n \log n).
\]
We now split into two subcases: $a > 2$ and $a \leq 2$.

For $a > 2$, we have
\[
    1+c_k' \geq \frac{an^2-(a-2)k^2}{2k(k+1)} = \frac{a}{2}\left[\frac{n^2-(1-\frac{2}{a})k^2}{k(k+1)}\right]
\]
and
\[
    1 + \frac{1}{c_k'} \geq \frac{an^2-(a-2)k^2}{a(n^2-k^2)} = \frac{n^2-(1-\frac{2}{a})k^2}{n^2-k^2} \geq 1.
\]
We first use
\[
    kc_k' = \frac{a(n^2-k^2)}{2(k+1)} \geq \frac{a}{2}(n-k),
\]
which implies the following, where $\xi = \sqrt{1 - \frac{2}{a}}$ for $a > 2$:
\[
    S \geq \sum_{k=1}^n k \log \frac{a(n+\xi k)(n-\xi k)}{2k(k+1)} + \frac{a}{2}\sum_{k=1}^{n-1} (n-k) \log \frac{(n+\xi k)(n-\xi k)}{(n+k)(n-k)}.
\]

Using $\xi = \sqrt{1 - \frac{2}{a}}$ and Corollary \ref{cor:quadratic-klogk} gives
\[
    \sum_{k=1}^n k \log \frac{a(n+\xi k)(n-\xi k)}{2k(k+1)} \geq \frac{n^2}{2} \left(\frac{a}{a-2}\log \left(\frac{a}{2}\right)\right) - O_\xi(n).
\]
Next, a straightforward computation gives
\[
    \frac{a}{2}\sum_{k=1}^{n-1} (n-k) \log \frac{(n+\xi k)(n-\xi k)}{(n+k)(n-k)} \geq \sum_{k=1}^{n-1} (n-k) \log \frac{n^2}{(n+k)(n-k)} = \sum_{k=1}^n k \log \frac{n^2}{k(2n-k)},
\]
since this expression is increasing in $a$ because $x \cdot \log(1+x^{-1})$ is increasing for $x \geq 0$. Using Lemma \ref{lem:klogk} we have
\[
    \sum_{k=1}^n k\log k = \frac{n^2}{2} \log n - \frac{n^2}{4} + \frac{n}{2}\log n \pm O(n),
\]
and
\[
    \sum_{k=1}^n k\log (2n-k) = \frac{n^2}{2} \log n + 2n^2 \log 2 - \frac{9n^2}{4} + n^2 + \frac{n}{2} \log n \pm O(n),
\]
and
and
\[
    \sum_{k=1}^n k \log (n^2) = n^2 \log n + n \log n.
\]
This implies
\[
    \sum_{k=1}^n k \log \frac{n^2}{k(2n-k)} \geq n^2\left(\frac{1}{4} - 2\log 2 + \frac{9}{4} - 1\right) - O(n) \geq \frac{n^2}{2} (3 - 4\log 2) - O(n).
\]
Combining everything then implies
\[
    S \geq \frac{n^2}{2} \left(3 - 4\log 2 + \frac{a}{a-2} \log\left(\frac{a}{2}\right)\right) - O_a(n),
\]
which bounds $S$ in the case of $ a > 2$. Note that the constant in this expression in front of $n^2$ approaches $2 - \log 2$ as $a \to 2$, which aligns with the bound for $a \leq 2$ below.

For $a \leq 2$, we instead have
\[
    1+c_k' \geq \frac{an^2-(a-2)k^2}{2k(k+1)} \geq \frac{a}{2}\left[\frac{n^2}{k(k+1)}\right]
\]
and
\[
    1 + \frac{1}{c_k'} \geq \frac{an^2-(a-2)k^2}{a(n^2-k^2)} \geq \frac{n^2}{n^2-k^2} \geq 1,
\]
which, for $b=\lfloor \sqrt{\frac{a}{2}} \cdot n\rfloor$ implies
\[
\begin{split}
    S &\geq \sum_{k=1}^b k \log \frac{an^2}{2k(k+1)} + \frac{a}{2}\sum_{k=1}^{n-1} (n-k) \log \frac{n^2}{(n+k)(n-k)} \\
        &= \sum_{k=1}^b k \log \frac{an^2}{2k(k+1)} + \frac{a}{2}\sum_{k=1}^n k \log \frac{n^2}{k(2n-k)}.
\end{split}
\]
Using Lemma \ref{lem:klogk} we have
\[
    \sum_{k=1}^b k\log k = \sum_{k=1}^b k\log(k+1) = \frac{an^2}{4} \log n + \frac{an^2}{8} \log\left(\frac{a}{2}\right) - \frac{an^2}{8} \pm O_a(n \log n),
\]
and
\[
    \sum_{k=1}^b k \log\left(\frac{an^2}{2}\right) \geq \frac{an^2}{2} \log n + \frac{an^2}{4} \log\left(\frac{a}{2}\right) \pm O_a(n \log n).
\]
This and the above bounds imply
\[
    S \geq \frac{an^2}{4} + \frac{an^2}{2}\left(\frac{1}{4} - 2\log 2 + \frac{9}{4} - 1\right) - O_a(n \log n) = a(1-\log 2) n^2 - O_a(n \log n),
\]
which bounds $S$ in the case that $a \leq 2$.

Combining the above then finally gives
\[
    \log K_n(\bm{N}) \geq S - 2n \log n - O_a(n) =
    \begin{cases}
        \frac{n^2}{2} \left(\frac{a}{a-2} \log\left(\frac{a}{2}\right) + 3 - 4\log 2\right) - O_a(n \log n), & a > 2 \\
        \frac{n^2}{2} \left(2a - 2a \log 2)\right) - O_a(n \log n), & a \leq 2
    \end{cases}.
\]

\subsubsection{The case of $p < 1$.}

Since $(n-x)^p(n+xp)$ is decreasing in $x$ for $x \in (0,n)$, we have for $k \leq \epsilon_n \cdot n$ that
\[
    (n-k)^p(n+kp) \geq n(n-\epsilon_n n)^p = (1-\epsilon_n)^p n^{1+p},
\]
where $\epsilon_n := \frac{1}{\log n}$ for all $n$. Letting $A_n := \frac{a(1-\epsilon_n)^p n^{p+1}}{p+1}$ and $c_k' := \frac{A_n}{(k+1)^2}$, we then have in this case that
\[
    c_k \geq \frac{a(n-k)^p(n+kp)}{k(k+1)(p+1)} \geq \frac{A_n}{k(k+1)} \geq c_k'.
\]
For all $x > 0$, we have $\frac{(x+1)^{x+1}}{x^x} \geq \left(\frac{1}{x}\right)^x$ and $\frac{(x+1)^{x+1}}{x^x} \geq ex + 1$ by Lemma \ref{lem:basic-entropy-bound}. With this, we have
\[
\begin{split}
    \log \prod_{k=1}^n \left(\frac{(c_k'+1)^{c_k'+1}}{(c_k')^{c_k'}}\right)^{k+1} &\geq \log\left[\prod_{k=1}^{\left\lfloor\sqrt{A_n}\right\rfloor} \left(e \cdot c_k'\right)^{k+1} \prod_{k=\left\lfloor\sqrt{A_n}\right\rfloor+1}^{\left\lfloor\epsilon_n \cdot n\right\rfloor} \left(\frac{1}{c_k'}\right)^{c_k'(k+1)}\right] \\
        &\geq \sum_{k=1}^{\left\lfloor\sqrt{A_n}\right\rfloor} (k+1)\left[\log(e A_n) - 2\log(k+1)\right] + \sum_{k=\left\lfloor\sqrt{A_n}\right\rfloor+1}^{\left\lfloor\epsilon_n \cdot n\right\rfloor} \frac{A_n\left[2\log(k+1) - \log A_n\right]}{k+1}.
\end{split}
\]
Note that $\epsilon_n \cdot n \geq \sqrt{A_n} \geq 2$ for $n$ large enough. Thus we have
\[
    \sum_{k=1}^{\left\lfloor\sqrt{A_n}\right\rfloor} (k+1)\log(e A_n) = \frac{A_n \log A_n}{2} + \frac{A_n}{2} \pm O(\sqrt{A_n} \log A_n),
\]
and using Lemma \ref{lem:klogk}, we have
\[
    \sum_{k=1}^{\left\lfloor\sqrt{A_n}\right\rfloor} -2(k+1)\log(k+1) = -\frac{A_n \log A_n}{2} + \frac{A_n}{2} \pm O(\sqrt{A_n} \log A_n),
\]
and using Lemma \ref{lem:euler-maclaurin} (with odd parameter $p=1$), we have
\[
    \sum_{k=\left\lfloor\sqrt{A_n}\right\rfloor+1}^{\epsilon_n \cdot n} 2 A_n \frac{\log(k+1)}{k+1} = A_n \log^2(\epsilon_n \cdot n) - \frac{A_n \log^2 A_n}{4} \pm O(A_n),
\]
and
\[
    \sum_{k=\left\lfloor\sqrt{A_n}\right\rfloor+1}^{\epsilon_n \cdot n} -\frac{A_n \log A_n}{k+1} = -A_n \log A_n \log(\epsilon_n \cdot n) + \frac{A_n \log^2 A_n}{2} \pm O(A_n \log A_n).
\]
Combining the above then gives
\[
    \log \prod_{k=1}^n \left(\frac{(c_k'+1)^{c_k'+1}}{(c_k')^{c_k'}}\right)^{k+1} \geq A_n \log^2\left(\frac{\epsilon_n \cdot n}{\sqrt{A_n}}\right) - O(A_n \log A_n)
\]
Using Lemma \ref{lem:basic-entropy-bound}, we then further compute
\[
    -2 \log \prod_{k=1}^n \frac{(c_k'+1)^{c_k'+1}}{(c_k')^{c_k'}} \geq -2 \log \prod_{k=1}^n e \left(c_k' + \frac{1}{2}\right) \geq -2n \log(e A_n).
\]
Combining everything and using Theorem \ref{thm:general_lower_bound} (2) and the fact that $x \mapsto \frac{(x+1)^{x+1}}{x^x}$ is increasing for $x > 0$ then gives
\[
    \log K_n({\bm{N}}) \geq \frac{a(1-p)^2}{4(p+1)} n^{p+1} \log^2 n - O(a n^{p+1} \log n \log\log n),
\]
which is the desired result.

\subsection{The Tesler $(1,1,\ldots,1,-n)$ case}

This case fits into the polynomial growth case, but we bound it more specifically here due to its importance. In this case, we have
\[
    N_k = 1, \quad s_k = k+1, \quad c_k = \frac{n+1}{k(k+1)}.
\]
Thus by Theorem \ref{thm:general_lower_bound} (1) we have
\[
    K_n(\bm{N}) \geq \frac{1}{n^2(n+1)^{n+1}} \prod_{k=1}^n \left(1 + \frac{n+1}{k(k+1)}\right)^k \prod_{k=1}^n \left(1 + \frac{k(k+1)}{n+1}\right)^{\frac{n+1}{k+1}}
\]
Note further that
\[
    \sum_{k=\lceil \sqrt{n+1} \rceil - 1}^n \frac{n+1}{k+1} \log\left(1 + \frac{k(k+1)}{n+1}\right) \geq \sum_{k=\lceil \sqrt{n+1} \rceil}^{n+1} \frac{n+1}{k} \log\left(\frac{k^2}{n+1}\right).
\]
We now use the following lemma.

\begin{lemma} \label{lem:unimodal-sum}
    Let $f: (a,b) \to \R$ be unimodal, and let $S = (a,b) \cap \Z$. Then
    \[
        \sum_{k \in S} f(k) \geq \int_a^b f(t) \, dt - \max_{t \in [a,b]} f(t).
    \]
\end{lemma}
\begin{proof}
    Let $t_0 \in [a,b]$ be such that $f(t_0)$ maximizes $f$ on $[a,b]$, and let $S_- := (a,t_0) \cap \Z$ and $S_+ := [t_0,b) \cap \Z$. Then,
    \[
        \sum_{k \in S_-} f(k) \geq \int_a^{\lceil t_0 \rceil} f(t) \, dt - f(t_0)
    \]
    and
    \[
        \sum_{k \in S_+} f(k) \geq \int_{\lceil t_0 \rceil}^b f(t) \, dt.
    \]
    Combining gives the desired result.
\end{proof}

Now consider the function $f(t): (\sqrt{n+1}, n+1) \to \R$, defined by
\[
    f(t) := \frac{n+1}{t} \log\left(\frac{t^2}{n+1}\right),
\]
which is unimodal with maximum $f(t_0) = \frac{2}{e} \sqrt{n+1}$ achieved at $t_0 = e \sqrt{n+1}$. Thus by Lemma \ref{lem:unimodal-sum}, we have
\[
    \sum_{k=\lceil \sqrt{n+1} \rceil}^{n+1} \frac{n+1}{k} \log\left(\frac{k^2}{n+1}\right) \geq \frac{n+1}{4} \log^2(n+1) - \frac{2}{e} \sqrt{n+1}.
\]
Note further that
\[
    \sum_{k=1}^n k \log\left(1 + \frac{n+1}{k(k+1)}\right) \geq \sum_{k=1}^n k \log\left(1 + \frac{1}{k}\right) \geq \sum_{k=1}^n k \left(\frac{1}{k} - \frac{1}{2k^2}\right) = n - \frac{1}{2} \sum_{k=1}^n \frac{1}{k}
\]
by standard Taylor series bounds. Standard harmonic series bounds then give
\[
    n - \frac{1}{2} \sum_{k=1}^n \frac{1}{k} \geq n - \frac{1}{2} \log n - 1.
\]
Combining everything gives
\[
    \log K_n(\bm{N}) \geq \frac{n+1}{4} \log^2(n+1) - (n+1) \log(n+1) + n - \frac{2}{e} \sqrt{n+1} - \frac{5}{2} \log(n) - 1.
\]
We now determine for which $n$ we have $\log K_n(\bm{N}) \geq (n-1) \log(n+1)$. First note that for $n \geq 4$ we have
\[
    2\log n + n - \frac{2}{e} \sqrt{n+1} - \frac{5}{2} \log(n) - 1 \geq 0,
\]
and thus in this case we have
\[
    \log K_n(\bm{N}) - (n-1) \log(n+1) \geq \frac{n+1}{4} \log^2(n+1) - 2(n+1) \log(n+1).
\]
A simple calculation then implies
\[
    \frac{n+1}{4} \log^2(n+1) - 2(n+1) \log(n+1) \geq 0
\]
for $n \geq e^8 - 1$, where $e^8 - 1 \leq 3000$.

\subsection{The $N_i = a \cdot n$ case}

We use the positive average of vertices from Proposition \ref{prop: center of mass positive netflow}, for which we have
\[
    N_k = a n, \quad s_k = a n(k+1), \quad c_k = \frac{a  n(n+1)}{k(k+1)} \geq \frac{a(n+1)^2}{(k+1)^2},
\]
which by Lemma \ref{lem:basic-entropy-bound} implies
\[
    \prod_{k=0}^{n-1} \frac{s_k^{s_k}}{(s_k+1)^{s_k+1}} \geq \prod_{k=0}^{n-1} \frac{2}{e}(2s_k + 1)^{-1} \geq \prod_{k=0}^{n-1} \frac{1}{ea(n+1)(k+1)} = \frac{1}{(ea)^n (n+1)^n n!}.
\]
Using Lemma \ref{lem:basic-entropy-bound} again, we then have
\[
    \frac{(c_k+1)^{c_k+1}}{c_k^{c_k}} \geq e \cdot c_k + 1 \geq ea \cdot \frac{(n+1)^2}{(k+1)^2},
\]
and thus
\[
    \prod_{k=1}^n \left(\frac{(c_k+1)^{c_k+1}}{c_k^{c_k}}\right)^k \geq \prod_{k=1}^n \left(ea \cdot \frac{(n+1)^2}{(k+1)^2}\right)^k = \left(ea(n+1)^2\right)^{\binom{n+1}{2}} \prod_{k=1}^n (k+1)^{-2k}.
\]
We then have
\[
    \binom{n+1}{2} \log\left(ea(n+1)^2\right) = (n+1)^2 \log(n+1) + \frac{(n+1)^2}{2} \log(ea) - (n+1) \log(n+1) - \frac{n+1}{2} \log(ea),
\]
and
\[
    -2 \sum_{k=1}^n k \log(k+1) \geq -(n+1)^2 \log(n+1) + \frac{(n+1)^2}{2} - (n+1) \log(n+1) - \frac{1}{6} \log(n+1) \pm O(1).
\]
Combining this and using Theorem \ref{thm:general_lower_bound} (1), we obtain
\[
    \log K_n(\bm{N}) \geq \frac{(n+1)^2}{2} (2 + \log(a)) - 4(n+1) \log(n+1) - O_a(n).
\]

\subsection{The $N_i = a \cdot n + i$ case}

We use the positive average of vertices from Proposition \ref{prop: center of mass positive netflow}, for which we have
\[
    c_k = \left(a + \frac{1}{2}\right) \cdot \frac{n(n+1)}{k(k+1)} - \frac{1}{2} \geq \left(a + \frac{1}{2}\right) \cdot \frac{(n+1)^2}{(k+1)^2} - \frac{1}{2} =: c_k'.
\]
Using Lemma \ref{lem:basic-entropy-bound}, we then have
\[
    \frac{(c_k'+1)^{c_k'+1}}{(c_k')^{c_k'}} \geq 2c_k' + 1 = (2 a + 1) \cdot \frac{(n+1)^2}{(k+1)^2},
\]
which implies
\[
\begin{split}
    \log \prod_{k=1}^n \left(\frac{(c_k'+1)^{c_k'+1}}{(c_k')^{c_k'}}\right)^{k-1} &\geq \sum_{k=1}^n 2(k-1) \log\left(\frac{n+1}{k+1}\right) + \sum_{k=1}^n (k-1) \log(2a + 1) \\
        &= \frac{n^2}{2} \left(1 + \log(2a + 1)\right) - O(n \log n)
\end{split}
\]

\subsection{The $N_i = n + i$ case}

The case fits into the more general $N_i = a \cdot n + i$ case, but we bound it more specifically here to compare it to Proposition \ref{prop:asympt n+delta}. We use the positive average of vertices from Proposition \ref{prop: center of mass positive netflow}, for which we have
\[
    c_k = \frac{3}{2} \cdot \frac{n(n+1)}{k(k+1)} - \frac{1}{2} \geq \frac{3}{2} \cdot \frac{(n+1)^2}{(k+1)^2} - \frac{1}{2} =: c_k'.
\]
Defining $\gamma := \frac{1}{6}$ and using Lemma \ref{lem:basic-entropy-bound}, we then have
\[
    \frac{(c_k'+1)^{c_k'+1}}{(c_k')^{c_k'}} \geq e\left(c_k' + \frac{1}{2} - \frac{1}{24c_k'}\right) \geq \frac{3e}{2} \left(\frac{(n+1)^4 - \gamma^2 (k+1)^4}{(n+1)^2(k+1)^2}\right)
\]
Thus we have
\[
    \log \prod_{k=1}^n \left(\frac{(c_k'+1)^{c_k'+1}}{(c_k')^{c_k'}}\right)^{k-1} \geq \sum_{k=1}^{n+1} (k-2) \left[\log\left(\frac{3e}{2(n+1)^2}\right) + \log((n+1)^2 + \gamma k^2) + \log\left(\frac{(n+1)^2 - \gamma k^2}{k^2}\right)\right].
\]
Since $f(t) = t \log((n+1)^2 + \gamma t^2)$ is increasing for $t > 0$, we have
\[
\begin{split}
    \sum_{k=1}^{n+1} k \log((n+1)^2 + \gamma k^2) &\geq \int_0^{n+1} t \log((n+1)^2 + \gamma t^2) \, dt \\
        &= \frac{1}{2\gamma} \left[((n+1)^2+\gamma t^2) (\log((n+1)^2+\gamma t^2) - 1)\right]_0^{n+1} \\   
        &\geq \frac{(n+1)^2}{2} \left[2 \log(n+1) - 1 + (1+\gamma^{-1}) \log(1+\gamma)\right]
\end{split}
\]
Further, using Corollary \ref{cor:quadratic-klogk} we have
\[
    \sum_{k=1}^{n+1} k \log\left(\frac{(n+1)^2 - \gamma k^2}{k^2}\right) \geq \frac{(n+1)^2}{2} (1-\gamma^{-1}) \log(1-\gamma) + \frac{n+1}{2} \log(1-\gamma) - \frac{1}{6} \log(n+1) \pm O_\gamma(1),
\]
and further we have
\[
    \sum_{k=1}^{n-1} k \log\left(\frac{3e}{2(n+1)^2}\right) = \left(\frac{(n+1)^2}{2} - \frac{3n+1}{2}\right) \left[1 + \log \frac{3}{2} - 2 \log(n+1)\right],
\]
and finally we also have
\[
    -2\sum_{k=1}^{n+1} \log\left(\frac{(n+1)^4 - \gamma^2 k^4}{k^2}\right) \geq -4(n+1) \log(n+1) - 4(n+1) + 2\log(n+1) + O(1).
\]
Combining everything gives
\[
    \log \prod_{k=1}^n \left(\frac{(c_k'+1)^{c_k'+1}}{(c_k')^{c_k'}}\right)^{k-1} \geq \frac{(n+1)^2}{2}\left[\log \frac{3}{2} + \log(1-\gamma^2) + \frac{1}{\gamma} \log\left(\frac{1+\gamma}{1-\gamma}\right)\right] - (n+1) \log(n+1) \pm O_\gamma(n).
\]
Since $\gamma = \frac{1}{6}$, this gives
\[
    \log K_n(\bm{N}) \geq 1.198 n^2 - O(n \log n).
\]
The coefficient of $n^2$ here is off from the correct coefficient given in Proposition \ref{prop:asympt n+delta} by about $0.1$.

\subsection{The $(t,0,0,\ldots,0,-t)$ case}

For the case that $\bm{N} = (t,0,\ldots,0,-t)$, which implies $s_k = t$ for all $k < n$, the average of the vertices is equal to 
\[
    M := t \cdot \begin{bmatrix}
        2^{-(n-1)} & 2^{-(n-1)} & 2^{-(n-2)} & \cdots & 2^{-3} & 2^{-2} & 2^{-1} \\
        2^{-(n-1)} & 2^{-(n-1)} & 2^{-(n-2)} & \cdots & 2^{-3} & 2^{-2} & 2^{-1} \\
        2^{-(n-2)} & 2^{-(n-2)} & 2^{-(n-3)} & \cdots & 2^{-2} & 2^{-1} & 0 \\
        \vdots & \vdots & \vdots & \ddots & \vdots & \vdots & \vdots \\
        2^{-3} & 2^{-3} & 2^{-2} & \cdots & 0 & 0 & 0 \\
        2^{-2} & 2^{-2} & 2^{-1} & \cdots & 0 & 0 & 0 \\
        2^{-1} & 2^{-1} & 0 & \cdots & 0 & 0 & 0
    \end{bmatrix},
\]
as seen in Proposition \ref{prop: center of mass CRY netflow}. Applying Theorem \ref{thm:matrix_lower_bound} gives
\[
    K_n(\bm{N}) \geq \left(\frac{t^t}{(1+t)^{1+t}}\right)^{2n-1} \cdot \frac{(t \cdot 2^{-(n-1)} + 1)^{t \cdot 2^{-(n-1)}+1}}{(t \cdot 2^{-(n-1)})^{t \cdot 2^{-(n-1)}}} \prod_{k=1}^{n-1} \left(\frac{(t \cdot 2^{-k} + 1)^{t \cdot 2^{-k}+1}}{(t \cdot 2^{-k})^{t \cdot 2^{-k}}}\right)^{n+2-k}.
\]
If $\log_2(et) \leq n-1$, we then have
\[
    \prod_{k=1}^{n-1} \left(\frac{(t \cdot 2^{-k} + 1)^{t \cdot 2^{-k}+1}}{(t \cdot 2^{-k})^{t \cdot 2^{-k}}}\right)^{n+2-k} \geq \prod_{k=1}^{n-1} \left(\frac{et}{2^k} + 1\right)^{n+2-k} \geq \prod_{k=1}^{\lfloor\log_2(et)\rfloor} 2^{(n+2-k)(\log_2(et) - k)},
\]
by Lemma \ref{lem:basic-entropy-bound}. We further have
\[
\begin{split}
    \sum_{k=1}^{\lfloor\log_2(et)\rfloor} (n+2-k)(\log_2(et) - k) &\geq (n+2) \left(\log_2^2(et) - \log_2(et)\right) - \left(n + 2 + \log_2(et)\right) \binom{\log_2(et) + 1}{2} \\
        &= \frac{n+2}{2} \left(\log_2^2(et) - 3\log_2(et)\right) - \frac{(\log_2(et) + 1) \log_2^2(et)}{2} \\
        &\geq \frac{n+2}{2} \cdot \log_2(et) \cdot \log_2\left(\frac{et}{8}\right) - \frac{1}{2} \log_2^3(et) - \frac{1}{2} \log_2^2(et),
\end{split}
\]
and using (\ref{eq:e_bounds}), we also have
\[
    \log_2\left(\left(\frac{t^t}{(1+t)^{1+t}}\right)^{2n-1}\right) \geq -(2n-1) \log_2(e(t+1)) \geq -2n \cdot \left(\log_2(et) + \frac{\log_2(e)}{t}\right) \geq -2n \cdot \log_2(et) - \frac{3n}{t}
\]
since $\log_2\left(1 + \frac{1}{t}\right) \leq \frac{\log_2(e)}{t}$. Combining everything then gives
\[
\begin{split}
    \log_2 K_n(\bm{N}) &\geq \frac{n+2}{2} \cdot \log_2(et) \cdot \log_2\left(\frac{et}{8}\right) - 2n \cdot \log_2(et) - \frac{3n}{t} - \frac{1}{2} \log_2^3(et) - \frac{1}{2} \log_2^2(et) \\
        &\geq \frac{n+2}{2} \cdot \log_2(et) \cdot \log_2\left(\frac{et}{128}\right) - \frac{3n}{t} - \frac{1}{2} \log_2^3(et) - \frac{1}{2} \log_2^2(et).
\end{split}
\]

\subsection{The $t\cdot 2\rho(n)$ case} \label{sec:rho-average}

For the case that $\bm{N} = t \cdot 2\rho(n)$ with $s_k = t (k+1)(n-k)$ for all $k$, we use the matrix described in (\ref{eq:2rho-midpoint}), given by
\[
    M := t \cdot \begin{bmatrix}
        1 & 1 & 1 & \cdots & 1 & 1 & 1 \\
        1 & 1 & 1 & \cdots & 1 & 1 & n-1 \\
        1 & 1 & 1 & \cdots & 1 & 2(n-2) & 0 \\
        \vdots & \vdots & \vdots & \ddots & \vdots & \vdots & \vdots \\
        1 & 1 & 1 & \cdots & 0 & 0 & 0 \\
        1 & 1 & 2(n-2) & \cdots & 0 & 0 & 0 \\
        1 & n-1 & 0 & \cdots & 0 & 0 & 0
    \end{bmatrix},
\]
where the subdiagonal entries are given by $k(n-k)$ for $1 \leq k \leq n-1$. Applying Theorem \ref{thm:matrix_lower_bound} gives
\[
    K_n(\bm{N}) \geq \max_k\left\{\frac{(1+s_k)^{1+s_k}}{s_k^{s_k}}\right\} \prod_{k=0}^{n-1} \left(\frac{s_k^{s_k}}{(1+s_k)^{1+s_k}}\right)^2 \left[\left(\frac{(1+t)^{1+t}}{t^t}\right)^{\binom{n+1}{2}} \prod_{k=1}^{n-1} \frac{(1+tk(n-k))^{1+tk(n-k)}}{(tk(n-k))^{tk(n-k)}}\right].
\]
To simplify this, note first that
\[
    \log \prod_{k=0}^{n-1} \frac{s_k^{s_k}}{(1+s_k)^{1+s_k}} \geq -2n \log n - n \log t - O(n)
\]

and
\[
    \log \prod_{k=1}^{n-1} \frac{(1+tk(n-k))^{1+tk(n-k)}}{(tk(n-k))^{tk(n-k)}} \geq 2(n-1)\log(n-1) + (n-1) \log t - O(n)
\]

and
\[
\log \max_k\left\{\frac{(1+s_k)^{1+s_k}}{s_k^{s_k}}\right\} \geq 2\log n + \log t - O(1)
\]
using Lemma \ref{lem:basic-entropy-bound} and Stirling's approximation. This implies
\[
\begin{split}
    \log K_n(\bm{N}) &\geq \binom{n+1}{2} \log\left(\frac{(1+t)^{1+t}}{t^t}\right) - 2n\log n - n \log t - O(n) \\
        &\geq \frac{n^2}{2} \log\left(\frac{(1+t)^{1+t}}{t^t}\right) - 2n\log n - O(n \log t).
\end{split}
\]

\subsection{Asymptotics via maximum flow entropy} \label{sec:asymptotics-max-flow-entropy}

In this section we finally prove Theorem \ref{thm:matrix-choice-asymptotics}. By Theorem \ref{thm:matrix_lower_bound}, we have
\[
    1 \geq \frac{\log K_n(\bm{N})}{\sup_{\bm{f} \in \mathcal{F}_n(\bm{N})} \mathcal{H}(\bm{f})} \geq 1 - \frac{2 \sum_{i=0}^{n-1} \log\left(\frac{(1 + s_i)^{1+s_i}}{s_i^{s_i}}\right)}{\sup_{\bm{f} \in \mathcal{F}_n(\bm{N})} \mathcal{H}(\bm{f})},
\]

where $s_i = \sum_{j=0}^i N_j$. Thus we only need to show that the denominator dominates the numerator in the right-most expression of the above inequality. By (\ref{eq:e_bounds}), we first have
\[
    2 \sum_{i=0}^{n-1} \log\left(\frac{(1 + s_i)^{1+s_i}}{s_i^{s_i}}\right) \leq 2n + 2 \sum_{i=0}^{n-1} \log(1 + s_i).
\]
By assumption we have that $N_i \leq a \cdot (i+1)^p$ for all $k$ for some fixed $a,p \geq 0$, which implies $s_i \leq a \cdot \frac{(i+2)^{p+1}}{p+1}$ for all $i$ by (\ref{eq:sum_p}). Thus
\[
    2 \sum_{i=0}^{n-1} \log\left(\frac{(1 + s_i)^{1+s_i}}{s_i^{s_i}}\right) \leq 2n + 2 \sum_{i=0}^{n-1} \log\left(1 + a \cdot \frac{(i+2)^{p+1}}{p+1}\right) \leq O_{a,p}(n) + 2 (p+1) \sum_{i=0}^{n-1} \log (i+2),
\]
which implies
\[
    2 \sum_{i=0}^{n-1} \log\left(\frac{(1 + s_i)^{1+s_i}}{s_i^{s_i}}\right) \leq O_{a,p}(n \log n).
\]
On the other hand, $N_i \geq 1$ for all $i$ implies
\[
    \sup_{\bm{f} \in \mathcal{F}_n(\bm{N})} \mathcal{H}(\bm{f}) \geq \log K_n(\bm{N}) \geq  \frac{n}{4} \log^2 n - O(n \log n)
\]
by Theorem \ref{thm:in-section-tesler} and Corollary \ref{cor:kpf inequality term-wise dominance}. This completes the proof.

\section{Bounds from the Lidskii lattice point formulas} \label{sec: lidskii}

In this section we use a known positive formula for $K_n(\bm{N})$ coming from the theory of flow polytopes to give bounds for $K_n(\bm{N})$ in different regimes than the rest of the paper. Specifically, we consider the case $\bm{N}=(t,0,\ldots,0,-t)$ where $t$ is much larger than $n$.

In the theory of flow polytopes, there is a positive formula for the number $K_n(\bm{N})$ of lattice of points  of $\mathcal{F}_n(\bm{N})$ called the {\em Lidskii formulas} due to Lidskii \cite{Lidskii} for the complete graph and for other graphs by Baldoni--Vergne \cite{BV} and Postnikov--Stanley \cite{Pitman_Stanley_1999}. See \cite{MM18,KMS} for proofs of these formulas via polyhedra subdivisions. Let  $\bm{\delta}:=(n-1,n-2,\ldots,1,0)$.

\begin{theorem}[{Lidskii formulas \cite[Proposition 39]{BV}}]
Let $\bm{N}$ $=(N_0,N_1,\ldots,N_{n-1},-\sum_i N_i)$ where each $N_i \in \mathbb{Z}_{\geq 0}$. Then the number $K_n(\bm{N})$ of lattice points of the flow polytope $\mathcal{F}_n(\bm{N})$ satisfy 
\begin{align}
    K_n{(\bm{N})} &= \sum_{\bm{j}} \binom{N_0+n-1}{j_0} \binom{N_1+n-2}{j_1}\cdots \binom{N_{n-1}}{j_{n-1}} \cdot K_n(\bm{j}-\bm{\delta}), \label{eq: lidskii lattice points1} 
\end{align}
where the sums are over weak compositions $\bm{j}=(j_0,j_1,\ldots,j_{n-1})$ of $\binom{n}{2}$.
\end{theorem}

The values $K_n(\bm{j}-\bm{\delta})$ of the Kostant partition function appearing on the formulas above are actually {\em mixed volumes} of certain flow polytopes.  
For $i=0,\ldots,n-1$, let 
\begin{equation} \label{eq:def Pi}
P_i:=\mathcal{F}_{n}(\underbrace{0,\ldots,0}_{i},1,0,\ldots,0,-1)
\end{equation}
For a composition $\bm{j}=(j_0,\ldots,j_{n-1})$, denote by $V(P_0^{j_0},\ldots,P_{n-1}^{j_{n-1}})$ the \emph{mixed volume} of $j_i$ copies of $P_i$.

\begin{theorem}[{\cite[Sec. 3.4]{BV}}] \label{cor: kpf are mixed volumes}
Let $\bm{N}$ $=(N_0,N_1,\ldots,N_{n-1},-\sum_i N_i)$ where each $N_i \in \mathbb{Z}_{\geq 0}$, then the flow polytope $\mathcal{F}_n(\bm{N})$ is the following Minkowski sum
\begin{equation} \label{eq: F_n as a Minkowski sum}
\mathcal{F}_n(\bm{N}) = N_0 P_0 + N_1 P_1 + \cdots + N_{n-1} P_{n-1},
\end{equation}
For a weak composition $\bm{j}=(j_0,\ldots,j_{n-1})$ of $\binom{n}{2}$ we have that 
\begin{equation}
K_n(\bm{j}-\bm{\delta}) = V(P_0^{j_0},\ldots,P_{n-1}^{j_{n-1}}).
\end{equation}
\end{theorem}

Since the Lidskii formula \eqref{eq: lidskii lattice points1} for $K_n(\bm{N})$ is nonnegative one could use it to bound $K_n(\bm{N})$. Let $\mathcal{S}^{+}_n(\bm{N})$ be the set of compositions $\bm{j}=(j_0,\ldots,j_{n-1})$ of $\binom{n}{2}$ with $\binom{N_0+n-1}{j_0} \binom{N_1+n-2}{j_1}\cdots \binom{N_{n-1}}{j_{n-1}} \cdot K_n(\bm{j}-\bm{\delta})>0$, and $S^+_n(\bm{N})=|\mathcal{S}^+_n(\bm{N})|$. Also, let 
\[
M_n(\bm{N}) := \max_{\bm{j}}\left\{\binom{N_0+n-1}{j_0} \binom{N_1+n-2}{j_1}\cdots \binom{N_{n-1}}{j_{n-1}} \cdot K_n(\bm{j}-\bm{\delta}) \mid \bm{j} \in \mathcal{S}^{+}_n(\bm{N}) \right\}.
\]

\begin{proposition} \label{prop: easy bounds lidskii lattice points}
    Let $\bm{N}$ $=(N_0,N_1,\ldots,N_{n-1},-\sum_i N_i)$ where each $N_i \in \mathbb{Z}_{\geq 0}$, then 
  \begin{equation} \label{eq:generic bounds from Lidskii}
S^+_n(\bm{N}) \cdot M_n(\bm{N}) \,\geq\, K_n({\bm{N}}) \,\geq\, M_n(\bm{N}).
\end{equation}  
\end{proposition}

\begin{proof}
From \eqref{eq: lidskii lattice points1} the total, $K_n(\bm{N})$, is at least the term with the largest contribution and at most the product of such a term and the number of terms.
\end{proof}

The bounds in \eqref{eq:generic bounds from Lidskii} are not so precise in the sense that for a given $\bm{N}$  it is unclear how to determine the $\bm{j}$ that yields the maximum $M_n(\bm{N})$. Also, some of the terms of \eqref{eq: lidskii lattice points1} vanish. Next, we give a characterization of the compositions in $\mathcal{S}_n^+(\bm{N})$.

\begin{proposition} \label{prop:char positive compositions in lidskii}
A weak composition $\bm{j}=(j_0,j_1,\ldots,j_{n-1})$ of $\binom{n}{2}$ is in $\mathcal{S}^+_n(\bm{N})$ if and only if  $j_i \leq N_i+n-i-1$ for $i=0,\ldots,n-1$ and $\bm{j} \unrhd (n-1,n-2,\ldots,1,0)$. 
\end{proposition}

\begin{proof}
The first restriction comes from the binomial coefficients in the product $\binom{N_0+n-1}{j_0} \binom{N_1+n-2}{j_1}\cdots \binom{N_{n-1}}{j_{n-1}}$. 

Next, we show that $K_n(\bm{j}-\bm{\delta})>0$ if and only if $\bm{j} \unrhd \bm{\delta}$. For the forward implication, if $K_n(\bm{j}-\bm{\delta})>0$ then the polytope $\mathcal{F}_n(\bm{j} -\bm{\delta})$ is nonempty. By the projection in Proposition~\ref{prop:projection to PS}, the polytope $PS_n(\bm{j} -\bm{\delta})$ is also nonempty. By definition of this polytope, being nonempty  implies that $\bm{j}- \bm{\delta} \unrhd  \bm{0}$.

For the converse, given a composition $\bm{j}=(j_0,\ldots,j_{n-1})$ of $\binom{n}{2}$ satisfying $\bm{j} \unrhd \bm{\delta}$ then $\bm{j}-\bm{\delta} \unrhd \bm{0}$. Thus by Proposition~\ref{prop:kpf inequality dominance} we have that $K_n(\bm{j}-\bm{\delta})\geq K_n(\bm{0})=1$ as desired.
\end{proof}

Next we look at a specific case like $\bm{N}=(t,0,\ldots,0,-t)$ for large $t$. Recall that an \emph{inversion} of a permutation $w=w_1w_2\cdots w_n$ of $n$ is a pair $(i,j)$ with $i<j$ and $w_j>w_i$. Let $I_{n,k}$ ($J_{n,k}$) be the number of permutations of $\{1,2,\ldots,n\}$ with (at most) $k$ inversions \cite[\href{https://oeis.org/A008302}{A008302},\href{https://oeis.org/A161169}{A161169}]{oeis}. In particular $J_{n,t}=n!$ for $t\geq \binom{n}{2}$. From standard facts about permutation enumeration \cite[Prop. 1.4.6]{EC1} and ``generatingfunctionology", these numbers have the following generating function for fixed $n$.
\[
\sum_{k=0}^{\binom{n}{2}} I_{n,k} q^k = [n]_q!,\quad \sum_{k=0}^{\infty} J_{n,k} q^k = \frac{[n]_q!}{1-q},
\]
where $[n]_q!=(1+q)(1+q+q^2)\cdots (1+\cdots+q^{n-1})$ is the $q$-analogue of $n!$. Note that $I_{n,k}=I_{n,\binom{n}{2}-k}$.

\begin{proposition} \label{prop: number of js tCRY case}
Let $\bm{N}=(t,0,\ldots,0,-t)$, then $S^{+}_n(\bm{N})=J_{n-1,t}$. In particular, for $t\geq \binom{n-1}{2}$ we have that $S^{+}_n(\bm{N})=(n-1)!$.

\end{proposition}

\begin{proof}
The number $S^{+}_n(\bm{N})$ counts  compositions $\bm{j}=(j_0,\ldots,j_{n-1})$ of $\binom{n}{2}$ satisfying 
\[
\bm{j} \unrhd \bm{\delta}, \text{ and } 
j_0 \leq t+n-1, \quad j_1 \leq n-2, \quad 
j_2 \leq n-3, \quad 
\ldots,\quad j_{n-1}\leq 0. 
\]
It suffices to show that there are $I_{n-1,\binom{n-1}{2}-t}=I_{n-1,t}$ such compositions with $j_0=t+n-1$ for $t=0,1,\ldots,\binom{n-1}{2}$. The first condition is implied by the others by the order reversing property of dominance order $\unrhd$. If $j_0=t+n-1$, then $(j_1,\ldots,j_{n-2})$ is a composition of $\binom{n-1}{2}$ satisfying $j_k \leq n-1-k$ for $k=1,\ldots,n-2$. Such compositions, viewed as inversion tables \cite[Prop. 1.3.12]{EC1}, are in bijection with permutations of $n-1$ with  $\binom{n-1}{2}-t$ inversions.
\end{proof}

Next, we determine the $M_n(\bm{N})$ for the case $\bm{N}=(t,0,\ldots,0,-t)$ for large enough $t$.

\begin{proposition} \label{prop: max j case tCRY}
   Let $\bm{N}=(t,0,\ldots,0,-t)$ for $t\geq n^3/2$ then  $M_n(\bm{N})=\binom{t+n-1}{\binom{n}{2}} \prod_{i=0}^{n-2} C_i$ which is achieved at $\bm{j}=(\binom{n}{2},0,\dots,0)$.
\end{proposition}

\begin{proof}
Consider the Minkowski sum decomposition in \eqref{eq: F_n as a Minkowski sum} of $\mathcal{F}_n(\bm{N})$.  By the definition of $P_i$ in \eqref{eq:def Pi}, this polytope is a translation of the face $Q_i$ of $P_0$:
\[
Q_i = \{ (f) \in P_0 \mid f_{j,j+1}=1, \text{ for } j=0,\ldots,i-1\}.
\]
From the mixed volume interpretation of $K_{n}(\bm{j}-\bm{\delta})$ in Theorem~\ref{cor: kpf are mixed volumes}, since $P_i$ is a translation of $Q_i\subseteq P_i$, and the fact that mixed volumes are monotically increasing then
\[
K_{n}({\textstyle \binom{n}{2}},0,\dots,0) = V(P_0^{\binom{n}{2}}) \geq V(P_0^{j_0},\ldots,P_{n-1}^{j_{n-1}})=K_n(\bm{j}-\bm{\delta}),
\]
for compositions $\bm{j}$ in $\mathcal{S}^+_n(\bm{N})$.
For $\bm{j}=(\binom{n}{2},0,\dots,0)$ by \eqref{eq:prod cat case} and the symmetry of the Kostant partition function by reversing the flows, we have that
\begin{equation} \label{eq:inequality kpfs}
K_n\left({\textstyle \binom{n-1}{2}},-n+2,-n+1,\ldots,-1,0\right) =  \prod_{i=0}^{n-2} C_i.
\end{equation}
Next, for $t \geq \frac{n^3}{2}$  one can show that 
\begin{equation} \label{eq:inequality binomials}
\binom{t+n-1}{\binom{n}{2}} \geq \binom{t+n-1}{j_0} \binom{n-2}{j_1}\cdots \binom{0}{j_{n-1}}.
\end{equation}
Indeed, suppose first $\bm{j}' = \bm{j} + e_i$ for some $i$, then
    \[
        \binom{n-2}{j_1} \cdots \binom{0}{j_{n-1}} \geq \frac{1}{n-2} \binom{n-2}{j'_1} \cdots \binom{0}{j'_{n-1}}.
    \]
    Next suppose $\ell \leq L$ and $\ell' = \ell-k \geq 0$, then
    \[
        \binom{L}{\ell} \geq \binom{L}{\ell'} \cdot \left(\frac{L-\ell}{\ell}\right)^k.
    \]
    Thus for $t \geq \frac{n^3}{2}$ and for any composition $\bm{j}$ in $\mathcal{S}^+_n(\bm{N})$ we have
    \[
        \frac{\binom{t+n-1}{\binom{n}{2}} \binom{n-2}{0}\cdots \binom{0}{0}}{\binom{t+n-1}{j_0} \binom{n-2}{j_1}\cdots \binom{0}{j_{n-1}}} \geq \left(\frac{\frac{n^3}{2}+n-1-\binom{n}{2}}{\binom{n}{2} (n-2)}\right)^{\binom{n}{2} - j_0} \geq 1,
    \]
as desired.

Combining, both inequalities \eqref{eq:inequality kpfs} and \eqref{eq:inequality binomials} we obtain that $M_n(\bm{N})$ has the desired value at ${\bm j}=(\binom{n}{2},0,\ldots,0)$.
\end{proof}

Putting the previous results together gives the main result of this section: bounds for $K_n(t,0,\ldots,0,-t)$ for large values of $t$.

\begin{corollary} \label{cor:bound kpf using Lidskii}
For $t\geq n^3/2$ we have that 
\[
(n-1)!\cdot \binom{t+n-1}{\binom{n}{2}} \prod_{i=0}^{n-2} C_i \,\geq\,   K_n{(t,0,\ldots,0,-t)}  \,\geq\, \binom{t+n-1}{\binom{n}{2}} \prod_{i=0}^{n-2} C_i. 
\]
\end{corollary}

\begin{proof}
The result follows by using Proposition~\ref{prop: easy bounds lidskii lattice points} for $\bm{N}=(t,0,\ldots,0,-t)$ and using both Propositions~\ref{prop: number of js tCRY case}, \ref{prop: max j case tCRY} to evaluate $M_n(t,0,\ldots,-t)$ and $S_n^+(t,0,\ldots,0,-t)$, respectively. 
\end{proof}

\begin{remark}
Note that the regime of $t$ in  Corollary~\ref{cor:bound kpf using Lidskii} is different from that of the rest of the paper where we assume that $t$ is constant with respect to $n$. It would be interesting to compare the upper bound above with the upper bound $F(t,n)=\prod_{1\leq i<j\leq n} \frac{2t+i+j-1}{i+j-1}$ in Proposition~\ref{prop: previous bounds CRY case}. Note that since in the regime $t\geq n^3/2$ the lower bound overwhelms $(n-1)!$, then the log of lower bound gives the correct asymptotics for $\log K_n(t,0,\ldots,0,-t)$.
\end{remark}

\section{Final remarks} \label{sec:final remarks}

\subsection{Integer flows of other graphs}

Let $G$ be a connected directed acyclic graph with $n+1$ vertices, let $\bm{N}=(N_0,\ldots,N_{n-1},-\sum_i N_i)$ with $N_i \in \mathbb{N}$ as before and denote by $K_G(\bm{N})$ the number of lattice points of  $\mathcal{F}_G(\bm{N})$. This number is also of interest for other graphs beyond the complete graph \cite{BV,MM18,Caracol}. It would be of interest to apply our methods and find bounds for $K_G(\bm{N})$. The polytope $\mathcal{F}_G(\bm{N})$ also projects to a face of a transportation polytope by zeroing out entries corresponding to missing edges in \eqref{eq: def injection flows to transportation matrix}. 

In the case when $\bm{N}=(1,0,\ldots,-1)$, since the associated flow polytope is integral and has no interior points then $K_G(1,0,\ldots,0,-1)$ is also the number of vertices of the polytope. The associated contingency tables $\phi(f_{ij})$ counted by $K_G(1,0,\ldots,0,-1)$ have marginals $\bm{\alpha}=\bm{\beta}=(1,\ldots,1)$ (see \eqref{eq: def injection flows to transportation matrix}), and so the entries of the tables are $0,1$.  In this case the associated polynomials are actually real stable, and thus stronger lower bounds are possible (see \cite{G15}).

There is also the following permanent and determinant formula for this number from \cite{juggling}.  

\begin{theorem}[{\cite[Thm. 6.16]{juggling}}]
\[
K_G(1,0,\ldots,0,-1) = \perm(M_G)=\det(N_G),
\]
where $M_G$ is the matrix with $(m_{ij})$ with $m_{i,i-1}=1$, $m_{ij}=1$ if $(i,j+1)$ is an edge of $G$ and $0$ otherwise. $N_G$ is the matrix $(n_{ij})$ with $n_{i,i-1}=-1$, $n_{ij}=1$ if $(i,j+1)$ is an edge of $G$ and $0$ otherwise.
\end{theorem}

Note that this permanent formula for $K_G(1,0,\ldots,0,-1)$ means we can also apply Gurvits' original capacity-based lower bound in \cite{G06}.

\subsection{Other capacity lower bounds}

The bounds in our paper rely on lower bounding the capacity of a multivariate power series. Recently, \cite{GKL23} gave lower bounds on the capacity of real stable polynomials to further improve upon the approximation factor for the metric traveling salesman problem (after the breakthrough work of \cite{KKO21}). That paper lower bounds $\cpc_{\bm\alpha}(p)$ based on how close the value of $\nabla \log p(\bm{1})$ is to $\bm\alpha$. The techniques of that paper and of our paper are completely different, and as of now we know of no connection between these techniques other than the goal of lower bounding the capacity in order to explicitly lower bound coefficients. That said, it is an open problem whether or not the techniques of \cite{GKL23} can be generalized to apply to (denormalized) Lorentzian polynomials (see Section 9 of \cite{GL21b}).

\subsection{Phase transitions in the polynomial growth case} \label{sec:fin-rem-poly-growth}

Do the phase transitions observed in the lower bounds of Theorem \ref{thm:general_positive_lower_bound} represent the actual nature of the number of integer flows, or are they simply an artifact of the proof? Already for the Tesler case, the best known upper bound on $\log K_n(\bm{N})$ is $O(n^2)$, and so it is possible that the phase transitions of the lower bounds are misleading. Phase transitions for the related problem of counting and random contingency tables with certain given marginals have been observed in \cite{LP22,DLP20} (predicted by \cite{Bar10}), but no analogous results have been proven for integer flows in the polynomial growth cases. We leave it as an open problem to improve or find corresponding upper bounds in these cases.

\subsection{A flow version of Barvinok's question for contingency tables}

In \cite[Eq. 2.3]{Bar07}, Barvinok asks the question of the general log-concavity of the number of contingency tables in terms of the marginal vectors. Concretely, let $T_S(\bm\alpha,\bm\beta)$ be the number of non-negative integer matrices with row sums $\bm\alpha$ and column sums $\bm\beta$ and support (non-zero entries) $S$. If $(\bm\alpha,\bm\beta) = \sum_{i=1}^k c_i (\bm\alpha^{(i)},\bm\beta^{(i)})$ is a convex combination of non-negative integer vectors, then is it always true that
\[
    T_S(\bm\alpha,\bm\beta) \geq \prod_{i=1}^k \left[T_S(\bm\alpha^{(i)},\bm\beta^{(i)})\right]^{c_i}?
\]
A version of this question can be asked specifically for non-negative integer flows, which gives a special case of the above question. This special case can be explicitly asked as follows. If $\bm{N} = \sum_{i=1}^k c_i \bm{N}^{(i)}$ is a convex combination of integer vectors summing to $\bm{0}$, then is it always true that
\[
    K_n(\bm{N}) \geq \prod_{i=1}^k \left[K_n(\bm{N}^{(i)})\right]^{c_i}?
\]
Finally, a different but related question is given as follows. Given $\bm{N},\bm{M}$ such that $\bm{N} \unrhd \bm{M}$ (that is, that $\bm{N}$ domniates $\bm{M}$; see Section~\ref{sec: previous bounds}), is it always true that
\[
    K_n(\bm{N}+\bm{M}) \geq K_n(\bm{N}) \cdot K_n(\bm{M})?
\]
See \cite{Bar07} for other similar questions and results for contingency tables.

\subsection{Case of the $q$-analogue of the Kostant partition function}

The function $K_n(\bm{N})$ has a known $q$-analogue by Lusztig \cite{Lusztig-qanalogue} that we denote by $K_n(\bm{N},q)$ and is defined as follows,
\[
K_n(\bm{N},q) := \sum_{f \in \cF_n(\bm{N}) \cap \mathbb{Z}^{\binom{n+1}{2}}} q^{|f|},
\]
where $|f|=\sum_{i,j} f_{ij}$. Alternatively, $K_n(\bm{N},q)$ is the coefficient of $\bm{z}^{\bm{N}}$ in the generating function
\[
\sum_{\bm{N}} K_n(\bm{N},q) \bm{z}^{\bm{N}} \,=\, \prod_{0\leq i<j \leq n} \frac{1}{1-qz_iz_j^{-1}},
\]
or via (\ref{eq: def injection flows to transportation matrix}) as the coefficient of $\bm{x}^{\bm\alpha} \bm{y}^{\bm\beta}$ (where $\bm\alpha,\bm\beta$ are defined as in Theorem \ref{thm:BLP-flow-bound}) in the generating function
\[
    \Phi'(\bm{x},\bm{y},q) = \prod_{\substack{0 \leq i,j \leq n-1 \\ i+j \leq n-1}} \frac{1}{1 - q x_i y_j} \prod_{\substack{0 \leq i,j \leq n-1 \\ i+j = n}} \frac{1}{1 - x_i y_j}.
\]
For fixed $q > 0$, $K_n(\bm{N},q)$ can be bounded via capacity bounds on $\Phi'(\bm{x},\bm{y},q)$ in a way similar to that of the results of this paper (i.e, via Theorem \ref{thm:BLP-flow-bound}). On the other hand, it is not clear how to adapt the results of this paper to bound or approximate the coefficients of $K_n(\bm{N},q)$. More specifically, the expression $1/(1-xyz)$ does not fit well into the context of this paper since there is no obvious way to adjust it to have the necessary log-concavity properties.

\subsection{Approximating volumes of flow polytopes} \label{sec:approx-vol-flows}

Beyond bounding the number of integer flows, we can also bound the volume of $\phi(\mathcal{F}_n(\bm{N}))$.
To do this, we adapt results from \cite{BLP}.
In particular, we can emulate the proof of Theorems 8.1 and 8.2 of \cite{BLP} to achieve the following bound:
\begin{align}
    \vol(\phi(\mathcal{F}_n(\bm{N}))) \,\geq\, \frac{f(S,n,n)}{e^{2n-1}} \, \max_i \{s_i\} \prod_{i=0}^{n-1} \frac{1}{s_i^2} \, \cpc_{\bm\alpha\,\bm\beta}\left(\prod_{i+j \leq n} \frac{-1}{\log(x_iy_j)}\right),
\end{align}
where we define

$S = \{(i,j) \in \{0,1,\ldots,n-1\}^2 : i+j \leq n\}$ and $f(S,n,n)^2$ is the covolume of the lattice $\Z\langle S \rangle \cap \R\langle \phi(\mathcal{F}_n(\bm{N})) \rangle$.
According to Section 8 of \cite{BLP}, $f(S,n,n)^2$ counts the number of spanning trees of the bipartite graph with support given by $S$. In our setting, the number of such trees is $(n!)^2$ (see Appendix~\ref{sec:number of spanning trees}).

The following analogue of Proposition \ref{prop:cap_sup} then follows from essentially the same proof.

\begin{proposition} \label{prop:cap_sup_vol}
   
    Let $\phi(\mathcal{F}_n(\bm{N}))$ be the image of $\mathcal{F}_n(\bm{N})$ in $\mathcal{T}(\bm\alpha,\bm\beta)$ as defined in (\ref{eq: def injection flows to transportation matrix}). We have that
    \[
        \cpc_{\bm\alpha \, \bm\beta}\left(\prod_{i+j \leq n} \frac{-1}{\log(x_iy_j)}\right) \,=\, e^{\binom{n}{2} + 2n - 1} \, \sup_{A \in \phi(\mathcal{F}_n(\bm{N}))} \prod_{i+j \leq n} a_{ij}.
    \]
\end{proposition}

This alternate expression and the above discussion allow us to produce concrete lower bounds for the volume of flow polytopes, which are analogous to our bounds on lattice points. If $\bm{f}$ is any (not necessarily integer) point of $\mathcal{F}_n(\bm{N})$ and $A = \phi(\bm{f})$ where $\phi$ is defined as in (\ref{eq: def injection flows to transportation matrix}), then the relative Euclidean volume of $\phi(\mathcal{F}_n(\bm{N}))$ can be bounded via
\begin{align}
    \vol(\phi(\mathcal{F}_n(\bm{N}))) \geq e^{\binom{n}{2}} n! \, \max_i \{s_i\} \prod_{i=0}^{n-1} \frac{1}{s_i^2} \prod_{i+j \leq n} a_{ij}.
\end{align}
Note that $\phi$ is an injective linear map, and thus bounds on the volume of $\mathcal{F}_n(\bm{N})$ can be obtained from the above volume bound. Further, we can use this to achieve specific volume bounds in a similar way to the flow counting bounds of Section \ref{sec:flow-counting-lower-bounds}. On the other hand, our vertex-averaging heuristic for choosing the matrix $A$ does not seem to work as well in the volume case.

Finally, note the difference in entropy functions within the supremums for counting and volume respectively in this paper is essentially the same as that of \cite{BH10} for counting and volume (see also \cite{Bar10}). These functions in these two cases are the entropy functions of the multivariate geometric and exponential distributions, respectively. And further, these distributions are entropy-maximizing distributions on the non-negative integer lattice and on the positive orthant, respectively.

\subsection{Projecting to a Pitman--Stanley polytope} \label{sec:PS polytope}

We settled Yip's conjecture (Conjecture~\ref{thm:Yip's conjecture}) for large enough $n$. However, Yip's original question was to find a projection from the polytope $\mathcal{F}_n(\bm{N})$ for $\bm{N}=(1,1\ldots,1,-n)$ and the classical permutahedron $\Pi_n$ that preserves lattice points of the latter. We were not able to find such a projection, however we were able to find the following projections of interest. 

For $\bm{a}=(a_1,\ldots,a_n) \in \mathbb{Z}^n$, let $PS_n(\bm{a})=\{(y_1,\ldots,y_n) \mid y_i\geq 0,  \sum_{i=1}^j y_i\leq \sum_{i=1}^j a_j, i=1,\ldots,n\}$ be the \emph{Pitman Stanley polytope} \cite{Pitman_Stanley_1999}. 

This polytope is a Minkowski sum of simplices \cite[Thm. 9]{Pitman_Stanley_1999} and is an example of a {\em generalized permutahedra}  \cite[Ex. 9.7]{AP}.

Baldoni--Vergne \cite[Ex. 16]{BV} showed that $PS_{n-1}(\bm{a})$ is integrally equivalent to a flow polytope of a graph $G_n$ with edges $\{(0,1),(1,2),\ldots,(n-2,n-1)\}\cup\{(0,n),(1,n),\ldots,(n-1,n)\}$ and netflow $(a_1,a_2,\ldots,a_n,-\sum_i a_i)$. Recall also that for $\bm{a}=\bm{1}$, then $PS_n(\bm{1})$ has $C_n$ lattice points and normalized volume $(n+1)^{n-1}$. The next result gives a projection between the flow polytope $\mathcal{F}_n(\bm{N})$ and the Pitman--Stanley polytope. The same projection appears in work of M\'eszaros--St. Dizier \cite[Thm. 4.9, Thm. 4.13, Ex. 4.17]{MStD} in the context of {\em saturated Newton polytopes} and generalied permutahedra\footnote{The projection considered by the authors \cite{MStD} allowed for other graphs $G$ other than the complete graph $k_{n+1}$ with a restricted netflow $\bm{N}$ depending on $G$.}.

\begin{proposition} \label{prop:projection to PS}
For $\bm{N}=(N_0,\ldots,N_{n-1},-\sum_{i=0}^{n-1} N_i)$ and $\bm{N}'=(N_0,\ldots,N_{n-2})$, the projection map $\pi:\mathcal{F}_n(\bm{N})\to \mathbb{R}^n$, $\bm{f}\mapsto (f_{0n},f_{1n},\ldots,f_{n-2,n})$ is a surjective map to $PS_{n-1}(\bm{N}')$  that preserves lattice points.     
\end{proposition}

\begin{proof}
 First we show that $\pi(\mathcal{F}_n(\bm{N}))\subset PS_{n-1}(
 \bm{N'})$. Given $\bm{f}$ in $\mathcal{F}_n(N)$,  since each flow in an edge is nonnegative, it suffices to check that $\sum_{i=0}^j f_{i,n} \leq \sum_{i=0}^j N_i$ for $j=0,\ldots,n-2$. Note that $\sum_{i=0}^j N_i$ equals the sum of the flows of the outgoing edges from vertices $\{0,\ldots,j\}$, i.e. all the flows on edges starting and ending in vertices $\{0,\ldots,j\}$ cancel. Thus
 \[
\sum_{i=0}^j N_i = \sum_{i=0}^j \sum_{k=j+1}^n f_{i,k} \geq \sum_{i=0}^j f_{i,n}.
 \]
To show the map is onto, we use the fact that $PS_{n-1}(\bm{N'})$ is integrally equivalent to $\mathcal{F}_{G_n}(\bm{N})$ and that $G_n$ is a subgraph of the complete graph $k_{n+1}$.  See Figure~\ref{fig:proj to PS}.  
\end{proof}

\begin{remark}
Restricting to the flows to the second to last vertex also gives a projection to the Pitman--Stanley polytope, see Figure~\ref{fig:symmetry proj to PS}.
\end{remark}

Lastly, restricting to the total outgoing flow of each vertex gives a projection to a parallelepiped.

\begin{proposition}
For $N=(N_0,\ldots,N_{n-1},-\sum_{i=0}^{n-1} N_i)$, the map $\pi'':\mathcal{F}(N) \to \mathbb{R}^{n-1}$, $\bm{f}\mapsto (x_1,\ldots,x_{n-1})$ where $x_i=\sum_{j=i+1}^n f_{ij}$ is a surjective map to the parallelepiped $[N_0,s_0]\times [N_1,s_1]\times [N_{n-1},s_{n-1}]$ where $s_i=\sum_{j=0}^i N_j$.
\end{proposition}

\begin{figure}
\begin{subfigure}[b]{0.45\textwidth}
    \centering
    \includegraphics[scale=0.8]{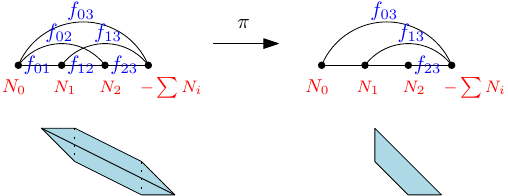}
    \caption{}
    \label{fig:proj to PS}
    \end{subfigure}
\begin{subfigure}[b]{0.5\textwidth}
     \raisebox{37pt}{\includegraphics[scale=0.8]{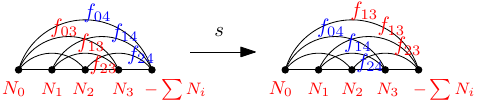}}
     \caption{}
      \label{fig:symmetry proj to PS}
\end{subfigure}
    \caption{(a) Projection from the flow polytope of the complete graph to the Pitman--Stanley polytope. 
    (b) Illustration symmetry map on $\mathcal{F}_n(N)$.
    }
    \label{fig:proj and symmetry to PS}
\end{figure}

\begin{proof}
From the netflow constraint on vertex $i$ we have that 
\[
N_i \leq \sum_{j=i+1}^n f_{ij} = N_i + \sum_{k=0}^i f_{ki}
\]
Since $N_i\geq 0$, the flow on each edge $(k,j)$ is at most the netflow $N_i$, that is $f_{kj} \leq N_k$. This shows that $\pi'(\mathcal{F}(N))\subseteq [N_0,s_0]\times [N_1,s_1]\times [N_{n-1},s_{n-1}]$. It is straightforward to check that this map is onto.
\end{proof}

\begin{remark}
The projections in this section do not give very strong lower bounds. For example for $\bm{N}=(1,1,\ldots,1,-n)$, they give the lower bounds $C_n$ and $n!$, respectively.   
\end{remark}

\section*{Acknowledgements}
Both authors acknowledge the support of the Natural Sciences and Engineering Research Council of Canada (NSERC), [funding reference number RGPIN-2023-03726 and RGPIN-2024-06246].

Cette recherche a été financée par le Conseil de recherches en sciences naturelles et en génie du Canada (CRSNG), [numéro de référence RGPIN-2023-03726 et RGPIN-2024-06246].

A. H. Morales was also partially supported by NSF grant DMS-2154019 and an FRQNT Team grant number 341288.

We thank the Institut Mittag-Leffler in
Djursholm, Sweden and the course of the program on Algebraic and Enumerative Combinatorics in Spring 2020 where the authors met. We thankfully acknowledge the support of the Swedish Research Council under grant no. 2016-06596, and thank Institut Mittag-Leffler for its hospitality. We thank Petter Br\"and\'en, Anne Dranowski, Leonid Gurvits, Joel Lewis, Jason O'Neill, Igor Pak, and Martha Yip for helpful comments and suggestions.

\bibliographystyle{amsalpha}
\bibliography{bibliography}

\appendix

\section{Summary of Basic Bounds and Asymptotics} \label{sec:bounds_asymptotics}

Throughout the arguments, we use a number of bounds and asymptotic expressions. We compile them here.

\begin{lemma} \label{lem:basic-entropy-bound}
    For all $t > 0$ we have
    \[
        e \left(t + \frac{1}{2}\right) \geq \frac{(t+1)^{t+1}}{t^t} \geq \max\left\{e \left(t + \frac{1}{2} - \frac{1}{24t}\right), \left(\frac{e}{t}\right)^t\right\}
    \]
    and
    \[
        \frac{(t+1)^{t+1}}{t^t} \geq et + 1.
    \]
    Note that this second bound is redundant with respect to the first bounds.
\end{lemma}
\begin{proof}
    We first prove the upper bound, which is equivalent to
    \[
        f(t) = 1 + \log\left(t + \frac{1}{2}\right) - (t+1) \log(t+1) + t \log t \geq 0
    \]
    for all $t > 0$. This holds in the limit as $t \to +\infty$, and thus it is sufficient to show that
    \[
        f'(t) = \frac{1}{t + \frac{1}{2}} - \log\left(1 + \frac{1}{t}\right) \leq 0
    \]
    for all $t > 0$. This holds in the limit as $t \to +\infty$, and thus it is sufficient to show that
    \[
        f''(t) = \frac{1}{t^2 + t} - \frac{1}{t^2 + t + \frac{1}{4}} \geq 0
    \]
    for all $t > 0$. This is clear, which proves the desired bound.

    We next show that
    \[
        f(t) = (t+1)\log(t+1) - t \log t - 1 - \log\left(t + \frac{1}{2} - \frac{1}{24t}\right) \geq 0
    \]
    for all $t > 0$. This clearly holds for $t < \frac{1}{24}$, so we will now prove it for $t \geq \frac{1}{24}$. This also holds in the limit as $t \to +\infty$, and thus it is sufficient to show that
    \[
        f'(t) = \log\left(1 + \frac{1}{t}\right) - \frac{24t^2 + 1}{t(24t^2 + 12t - 1)} \leq 0
    \]
    for $t \geq \frac{1}{24}$. This holds in the limit as $t \to +\infty$, and thus it is sufficient to show that
    \[
        f''(t) = \frac{144 t^2 + 22 t - 1}{t^2 (t + 1) (24 t^2 + 12 t - 1)^2} \geq 0
    \]
    for $t \geq \frac{1}{24}$. The denominator is positive, and the quadratic numerator is positive at $t = \frac{1}{24}$ and negative at $t=0$. Thus the above inequality holds for $t \geq \frac{1}{24}$, proving the desired bound.

    We next show that
    \[
        (t+1) \log(t+1) - t \geq 0
    \]
    for $t > 0$. This holds in the limit as $t \to 0^+$, and thus it is sufficient to show that
    \[
        f'(t) = \log(t+1) \geq 0
    \]
    for all $t > 0$. This is immediate, which proves the desired bound.

    Finally we show that
    \[
        f(t) = \frac{(t+1)^{t+1}}{t^t} - et - 1 \geq 0
    \]
    for all $t > 0$. This holds in the limit at $t \to 0^+$, and thus it is sufficient to show that
    \[
        f'(t) = \log\left(1 + \frac{1}{t}\right) \cdot \frac{(t+1)^{t+1}}{t^t} - e \geq 0
    \]
    for $t > 0$. This holds in the limit as $t \to +\infty$, and thus it is sufficient to show that
    \[
        f''(t) = \left[\log^2\left(1 + \frac{1}{t}\right) + \frac{1}{t+1} - \frac{1}{t}\right] \frac{(t+1)^{t+1}}{t^t} \leq 0
    \]
    for all $t > 0$. This is equivalent to showing that
    \[
        g(s) = \log^2(1 + s) + \frac{s}{1+s} - s \leq 0
    \]
    for all $s = \frac{1}{t} > 0$. This holds for $s = 0$, and thus it is sufficient to show that
    \[
        g'(s) = \frac{2 \log(1+s) + \frac{1}{1+s} - (1+s)}{1+s} = \frac{h(s)}{1+s} \leq 0
    \]
    for all $s > 0$. Since $h(0) = 0$, it is thus sufficient to show that
    \[
        h'(s) = -\left(1 - \frac{1}{1+s}\right)^2 \leq 0
    \]
    for all $s > 0$. This is immediate, which proves the desired bound.
\end{proof}

\subsection*{Bounds on $e$.} For all $x > 0$ we have
\begin{equation} \label{eq:e_bounds} \tag{\ref*{sec:bounds_asymptotics}.B1}
    \left(\frac{x}{x+1}\right)^x \geq \frac{1}{e} \qquad \text{and} \qquad \left(\frac{x+1}{x}\right)^{x+1} \geq e.
\end{equation}

\subsection*{Bounds on product of Catalan numbers.}

This is from \cite[Lemma 7.4]{MoralesShi}.
\begin{equation} \label{eq:catalan} \tag{\ref*{sec:bounds_asymptotics}.B2}
    \log\left(\prod_{i=1}^{n-1} C_i\right) = n^2\log 2 -\frac32 \log n +O(n).
\end{equation}

\subsection*{Other asymptotic expression and bounds.} We recall various other asymptotic expressions and bounds. The bounds here are obtainable using simple integration approximation or a standard application of the Euler-Maclaurin formula (Lemma \ref{lem:euler-maclaurin}).
\begin{equation} \label{eq:stirling} \tag{\ref*{sec:bounds_asymptotics}.B3}
    \sqrt{2\pi n} \cdot \frac{n^n}{e^n} \leq n! \leq e\sqrt{n} \cdot \frac{n^n}{e^n}
\end{equation}
\begin{equation} \label{eq:sum_p} \tag{\ref*{sec:bounds_asymptotics}.B4}
    \frac{k^{p+1}}{p+1} \leq \sum_{j=0}^k j^p \leq \frac{(k+1)^{p+1}}{p+1} \qquad \text{for } p \geq 0
\end{equation}
\begin{equation} \label{eq:harmonic} \tag{\ref*{sec:bounds_asymptotics}.B5}
    \log(n+1) \leq \sum_{k=1}^n \frac{1}{k} \leq 1 + \log n \qquad \text{and} \qquad \sum_{k=1}^n \frac{1}{k} = \log n + O(1),
\end{equation}
\begin{align} \label{eq:klogk} \tag{\ref*{sec:bounds_asymptotics}.B6}
    \frac{n^2\log n}{2} - \frac{n^2}{4} + \frac{1}{4} \leq &\sum_{k=1}^n k \log k \leq \frac{(n+1)^2\log(n+1)}{2} - \frac{(n+1)^2}{4} + \frac{1}{4} \qquad \text{and} \\
        &\sum_{k=1}^n k \log k = \frac{n^2\log n}{2} - \frac{n^2}{4} + \frac{n\log n}{2} + \frac{\log n}{12} + O(1), \nonumber
\end{align}
\begin{align} \label{eq:logk_k} \tag{\ref*{sec:bounds_asymptotics}.B7}
    \frac{\log 2}{2} + \frac{\log^2 (n+1) - \log^2 3}{2} \leq &\sum_{k=1}^n \frac{\log k}{k} \leq \frac{\log 2}{2} + \frac{\log 3}{3} + \frac{\log^2 n - \log^2 3}{2} \qquad \text{and} \\
        &\sum_{k=1}^n \frac{\log k}{k} = \frac{\log^2 n}{2} + O(1). \nonumber
\end{align}
% sum_1^n (ln k)/k > (ln 2)/2 \int_3^{n+1} (ln x)/x dx
%                  < (ln 2)/2 + (ln 3)/3 + \int_3^n (ln x)/x dx

% \int (ln x)/x dx = (1/2) (ln x)^2

\subsection*{Euler-Maclaurin formula.}

All of the above bounds can be derived from the Euler-Maclaurin formula, stated below. We make heavy use of Lemma \ref{lem:klogk}, which is a straightforward corollary.

\begin{lemma}[Euler-Maclaurin formula; e.g. see \cite{EulerMaclaurin}] \label{lem:euler-maclaurin}
    Given integers $a < b$, a positive odd integer $p$, and a smooth function $f: [a,b] \to \R$ we have
    \[
        \left|\sum_{k=a}^b f(k) - \left[\int_a^b f(t)\, dt + \frac{f(b) + f(a)}{2} + \sum_{k=1}^{\frac{p-1}{2}} \frac{B_{2k}}{(2k)!} \left(f^{(2k-1)}(b) - f^{(2k-1)}(a)\right)\right]\right| \leq \frac{2 \cdot \zeta(p)}{(2\pi)^{p}} \int_a^b |f^{(p)}(t)| dt,
    \]
    where $B_k$ is the $k^\text{th}$ Bernoulli number and $\zeta$ is the Riemann zeta function.
\end{lemma}

\begin{lemma} \label{lem:klogk}
    Let $a<b$ be integers, and suppose $c,d$ are real with $r := -\frac{d}{c}$ such that $ct+d > 0$ for all $t \in [a,b]$. Letting $\gamma := \min\{|b-r|,|a-r|\}$, we have
    \[
    \begin{split}
        \sum_{k=a}^b k \log(ck+d) &= \frac{b^2}{2} \log|b-r| - \frac{a^2}{2} \log|a-r| + \frac{r^2}{2} \log\left|\frac{a-r}{b-r}\right| \\
            &+ \frac{b^2 \log|c|}{2} - \frac{a^2\log|c|}{2} - \frac{(b+r)^2}{4} + \frac{(a+r)^2}{4} \\
            &+ \frac{b}{2} \log|b-r| + \frac{a}{2} \log|a-r| + \frac{(a+b)\log|c|}{2} - \frac{1}{12} \log\left|\frac{a-r}{b-r}\right| + O\left(\frac{|r|+1}{\gamma} + \frac{|r|}{\gamma^2}\right).
    \end{split}
    \]
    % \[
    %     \left|g(r) - \sum_{k=a}^b k \log(ck+d)\right| \leq h(r)
    % \]
    % where
    % \[
    % \begin{split}
    %     g(x) &= \frac{(2b+1)^2}{8} \log|b-r| - \frac{(2a-1)^2}{8} \log|a-r| + \frac{b(b+1)-a(a-1)}{2} \log|c| \\
    %         &- \frac{(b+r)^2}{4} + \frac{(a+r)^2}{4} + \frac{12r^2 + 1}{24} \log \left(\frac{a-r}{b-r}\right) - \frac{r(b-a)}{12(a-r)(b-r)}
    % \end{split}
    % \]
    % and
    % \[
    %     h(x) = \frac{2 \zeta(3)}{(2\pi)^3} \left(\frac{2(b-a)}{(r-b)(r-a)} + \frac{(b-a)|r^2 - ab|}{(r-b)^2(r-a)^2}\right).
    % \]
\end{lemma}
\begin{proof}
    We compute the approximation of the sum $\tilde{S}$, given by Lemma \ref{lem:euler-maclaurin} with $p=3$ and $f(t) = t \log(ct+d)$. We have
    \[
        \tilde{S} = \int_a^b f(t)\, dt + \frac{f(a)+f(b)}{2} + \frac{f'(b) - f'(a)}{12},
    \]
    where
    \[
        \int_a^b f(t)\, dt = \left[\frac{(t^2-r^2) \log|t-r|}{2} + \frac{t^2 \log|c|}{2} - \frac{(t+r)^2}{4} + \frac{r^2}{4}\right]_a^b,
    \]
    and
    \[
        \frac{f(b)+f(a)}{2} = \frac{b \log|b-r| + a \log|a-r| + (a+b)\log|c|}{2},
    \]
    and
    \[
        \frac{f'(b) - f'(a)}{12} = \frac{1}{12}\left[\log|t-r| + \log|ec| + \frac{r}{t-r}\right]_a^b = -\frac{1}{12} \log\left|\frac{a-r}{b-r}\right| + \frac{r}{12(b-r)} - \frac{r}{12(a-r)},
    \]
    and
    \[
        \int_a^b |f^{(3)}(t)|\, dt = \left[-\frac{1}{t-r}\right]_a^b + \left|\left[\frac{r}{(t-r)^2}\right]_a^b\right| = -\frac{1}{b-r} + \frac{1}{a-r} + \left|\frac{r}{(b-r)^2} - \frac{r}{(a-r)^2}\right|.
    \]
    We then have
    \[
        \frac{r}{12(b-r)} - \frac{r}{12(a-r)} + \frac{2 \zeta(3)}{(2\pi)^3} \int_a^b |f^{(3)}(t)|\, dt = O\left(\frac{|r|+1}{\gamma} + \frac{|r|}{\gamma^2}\right).
    \]
    Combining everything and simplifying yields the result.
\end{proof}

\begin{corollary} \label{cor:quadratic-klogk}
    For any fixed $\xi \in [0,1)$, we have
    \[
        \sum_{k=1}^n k \log \frac{n^2-\xi^2 k^2}{k^2} = \frac{n^2}{2} \left(\frac{1}{\xi^2} - 1\right) \log \left(\frac{1}{1-\xi^2}\right) + \frac{n}{2} \log(1-\xi^2) - \frac{1}{6} \log n \pm O_\xi(1),
    \]
    where $\lim_{\xi \to 0} \left(\frac{1}{\xi^2} - 1\right) \log \left(\frac{1}{1-\xi^2}\right) = 1$.
\end{corollary}
\begin{proof}
    Using Lemma \ref{lem:klogk}, we compute (with parameters $a=0,b=n,c=-\xi,r=\frac{n}{\xi}$)
    \[
        \sum_{k=1}^n k \log (n - \xi k) = \frac{n^2}{2} \log n + \frac{n^2}{2} \left[\left(1 - \frac{1}{\xi^2}\right) \log (1-\xi) - \frac{1}{2} - \frac{1}{\xi}\right] + \frac{n}{2} \log n + \frac{n}{2} \log(1-\xi) \pm O_\xi(1),
    % \begin{split}
    %     \sum_{k=1}^n k \log (n - \xi k) &= \frac{n^2}{2} \log(n(\frac{1}{\xi}-1)) + \frac{n^2}{2\xi^2} \log\frac{1}{1-\xi} + \frac{n^2}{2} \log \xi - \frac{n^2(1 + \frac{1}{\xi})^2}{4} + \frac{n^2}{4\xi^2} + \frac{n}{2} \log(n(\frac{1}{\xi}-1)) + \frac{n}{2} \log \xi \pm O_\xi(1) \\
    %         &= \frac{n^2}{2} \log n + \frac{n^2}{2} \left[\left(1 - \frac{1}{\xi^2}\right) \log (1-\xi) - \frac{1}{2} - \frac{1}{\xi}\right] + \frac{n}{2} \log n + \frac{n}{2} \log(1-\xi) \pm O_\xi(1),
    % \end{split}
    \]
    and (with parameters $a=0,b=n,c=\xi,r = -\frac{n}{\xi}$)
    \[
        \sum_{k=1}^n k \log (n + \xi k) = \frac{n^2}{2} \log n + \frac{n^2}{2} \left[\left(1 - \frac{1}{\xi^2}\right) \log (1+\xi) - \frac{1}{2} + \frac{1}{\xi}\right] + \frac{n}{2} \log n + \frac{n}{2} \log(1+\xi) \pm O_\xi(1),
    % \begin{split}
    %     \sum_{k=1}^n k \log (n + \xi k) &= \frac{n^2}{2} \log(n(\frac{1}{\xi}+1)) + \frac{n^2}{2\xi^2} \log\frac{1}{1+\xi} + \frac{n^2}{2} \log \xi - \frac{n^2(1 - \frac{1}{\xi})^2}{4} + \frac{n^2}{4\xi^2} + \frac{n}{2} \log(n(\frac{1}{\xi}+1)) + \frac{n}{2} \log \xi \pm O_\xi(1) \\
    %         &= \frac{n^2}{2} \log n + \frac{n^2}{2} \left[\left(1 - \frac{1}{\xi^2}\right) \log (1+\xi) - \frac{1}{2} + \frac{1}{\xi}\right] + \frac{n}{2} \log n + \frac{n}{2} \log(1+\xi) \pm O_\xi(1),
    % \end{split}
    \]
    and (with parameters $a=1,b=n,c=1,r=0$)
    \[
        \sum_{k=1}^n k \log k = \frac{n^2}{2} \log n - \frac{n^2}{4} + \frac{n}{2} \log n + \frac{1}{12} \log n \pm O(1).
    \]
    This gives
    \[
        \sum_{k=1}^n k \log \frac{n^2-\xi^2 k^2}{k^2} = \frac{n^2}{2} \left(\frac{1}{\xi^2} - 1\right) \log \left(\frac{1}{1-\xi^2}\right) + \frac{n}{2} \log(1-\xi^2) - \frac{1}{6} \log n \pm O_\xi(1).
    \]
    Note that a straightforward argument gives the same expression in the limit when $\xi = 0$.
\end{proof}

\section{Number of spanning trees} \label{sec:number of spanning trees}

For a tree $T$ with vertex set $V$, let $m(T) = \prod_{v\in V} x_v^{deg_T(v)-1}$. let $t_G$ be the following multivariate sum over spanning trees.
\[
t_G(x_1,\ldots,x_n) = \sum_T m(T).
\]
Let $G(n)\subset K_{n,n}$ be the bipartite graph with edges $(i,n+j)$ if $i-j<2$. 

\begin{theorem}[{\cite[Thm. 2.1]{Ehrenborg_vWilligenburg}}]
Let $\lambda \subset m\times n$ with $\lambda_1=n$. Let $G(\lambda)$ be a bipartite graph with vertices $\{1,\ldots,m\} \cup \{m+1,\ldots,m+n\}$ and edges $(i,m+j)$ if $(i,j) \in [\lambda]$, then
\[
t_{G(\lambda)}(x_1,\ldots,x_m; y_1,\ldots,y_n) \,=\, \prod_{i=2}^m (x_1+\cdots+x_{\lambda_i}) \prod_{j=2}^n (y_1+\cdots+y_{\lambda'_j}),
\]
in particular $G(\lambda)$ has $\prod_i \lambda_i \prod_j \lambda'_j$ spanning trees.
\end{theorem}

\begin{corollary}
For the graph $G(n)$ we have that 
\[
t_{G(n)}(x_1,\ldots,x_n;y_1,\ldots,y_n) \,=\, (x_1+x_2)(x_1+x_2+x_3)\cdots (x_1+\cdots + x_n)\cdot (y_1+y_2)(y_1+y_2+y_3)\cdots (y_1+\cdots + y_n),
\]
in particular $G(n)$ has $(n!)^2$ spanning trees.
\end{corollary}

\end{document}